\documentclass[10pt,journal,compsoc]{IEEEtran}
\ifCLASSOPTIONcompsoc
  \usepackage[nocompress]{cite}
\else
  \usepackage{cite}
\fi
\ifCLASSINFOpdf
\else
\fi
\newcommand{\cO}{\mathcal{O}}

\newcommand{\EE}{\mathbb{E}}
\newcommand{\RR}{\mathbb{R}}

\newcommand{\xb}{\mathbf{x}}
\newcommand{\zb}{\mathbf{z}}
\newcommand{\wb}{\mathbf{w}}

\newcommand{\bbx}{\bar{x}}

\newcommand{\vb}{\mathbf{v}}

\newcommand{\bs}{\mathbf{s}}
\newcommand{\bbs}{\bar{s}}

\newcommand{\Ab}{\mathbf{A}}

\newcommand{\argmin}{\mathop{\mathrm{argmin}}}

\newcommand{\norm}[1]{\left\|#1\right\|}
\newcommand{\dotprod}[1]{\left\langle #1\right\rangle}

\newcommand{\proximal}{\textbf{prox}}
\usepackage{amssymb}
\usepackage{amsmath}
\usepackage{amsmath,amssymb,amsfonts,amsthm}

\newtheorem{lemma}{Lemma}
\newtheorem{theorem}{Theorem}

\newtheorem{remark}{Remark}

\ifx\assumption\undefined
\newtheorem{assumption}{Assumption}

\usepackage{tabularx,booktabs}

\usepackage{ifpdf}
 \ifpdf
 \else
 \fi

\usepackage{cite}

\ifCLASSINFOpdf
   \usepackage[pdftex]{graphicx}
   \graphicspath{{../pdf/}{../jpeg/}}
   \DeclareGraphicsExtensions{.pdf,.jpeg,.png}
\else
   \usepackage[dvips]{graphicx}
   \graphicspath{{../eps/}}
   \DeclareGraphicsExtensions{.eps}
\fi

\usepackage{algorithmic}
\usepackage{algorithm}

\usepackage{array}

\ifCLASSOPTIONcompsoc 
  \usepackage[caption=false,font=normalsize,labelfont=sf,textfont=sf]{subfig}
\else
  \usepackage[caption=false,font=footnotesize]{subfig}
\fi

\usepackage{url}
\usepackage{float}
\makeatletter
\let\NAT@parse\undefined
\makeatother
\usepackage{hyperref}  

\allowdisplaybreaks[3]

\makeatletter
\let\NAT@parse\undefined
\makeatother
\usepackage{hyperref}  
\begin{document}
%
\title{PMGT-VR: A decentralized proximal-gradient algorithmic framework with variance reduction}
%
%
%
%

\author{Haishan~Ye*,
        Wei~Xiong*,
        and~Tong~Zhang,~\IEEEmembership{Fellow,~IEEE}
\IEEEcompsocitemizethanks{
\IEEEcompsocthanksitem Equal Contribution.
\IEEEcompsocthanksitem Haishan Ye is with Shenzhen Research Institute of Big Data; The Chinese University of Hong Kong, Shenzhen, email: hsye cs@outlook.com.
\IEEEcompsocthanksitem Wei Xiong is with Department of Probability and Statistics; University of Science and Technology of China, email:weixiong5237@gmail.com.
\IEEEcompsocthanksitem Tong Zhang is with Hong Kong University of Science and Technology, email: tongzhang@ust.hk.}

}

\IEEEtitleabstractindextext{%
\begin{abstract}
This paper considers the decentralized composite optimization problem. We propose a novel decentralized variance reduction proximal-gradient algorithmic framework, called \texttt{PMGT-VR}, which is based on a combination of several techniques including multi-consensus, gradient tracking, and variance reduction. The proposed framework relies on an imitation of centralized algorithms and we demonstrate that algorithms under this framework achieve convergence rates similar to that of their centralized counterparts. We also describe and analyze two representative algorithms,  \texttt{PMGT-SAGA} and \texttt{PMGT-LSVRG}, and compare them to existing state-of-the-art proximal algorithms. To the best of our knowledge, \texttt{PMGT-VR} is the first linearly convergent decentralized stochastic algorithm that can solve decentralized composite optimization problems.  Numerical experiments are provided to demonstrate the effectiveness of the proposed algorithms.
\end{abstract}

\begin{IEEEkeywords}
Decentralized optimization, proximal-gradient, variance reduction.
\end{IEEEkeywords}}

\maketitle

\IEEEdisplaynontitleabstractindextext

%
\IEEEpeerreviewmaketitle

\IEEEraisesectionheading{\section{Introduction}\label{sec:introduction}}

%
%
%
%
\IEEEPARstart{F}{or} modern large-scale optimization problems, the distributed computing architectures based algorithms have recently attracted significant attention and have been studied extensively in machine learning, control, and optimization communities. 
There are two main distributed settings that have been widely studied, namely, the master/slave setting and the decentralized setting. In the master/slave setting, a parameter server \cite{parameterserver2014muli} aggregates the local gradients computed by all other agents and performs the update. In the decentralized setting, agents connected by a network are only permitted to communicate with their neighbors to cooperatively solve the optimization problem. We further illustrate them in Figure~\ref{fig:topology}. 

In the context of distributed optimization,
the decentralized setting has long been treated as a compromise when a centralized topology is unavailable or the decentralization is natural. However, there has been a growing interest in the decentralized setting recently. Several reasons account for this phenomenon: (a) decentralized setting is of a lower communication cost at the busiest agent as this agent only communicates with her neighbors instead of all other agents \cite{dsgd2017lian}; (b) theoretically, several works have shown that we could design decentralized algorithms with a similar convergence rate as compared to their centralized counterparts; and (c) in practical, efficient topologies have been proposed to achieve most communication efficiency and numerical results are provided to give a positive answer to the question whether \textit{decentralized algorithms can be faster than its centralized counterpart} \cite{bluefrog2021, dsgd2017lian}. In this paper, we will focus on the algorithms design.

In the context of first-order methods, with a recently introduced gradient tracking technique, \texttt{GT-DGD} proposed in \cite{qu2017harnessing} and \texttt{GT-DSGD} proposed in \cite{shi2018gtdsgd} (decentralized gradient descent and stochastic gradient descent with gradient tracking) achieve similar convergence rates with \texttt{CGD} and \texttt{CSGD} (centralized \texttt{GD} and \texttt{SGD}), respectively. Efforts have also been made in developing the decentralized versions of variance reduction algorithm. However, existing works including \cite{xin2020variance}, \cite{gt_sarah2020shicong}, and \cite{li2020optimal} either achieve a convergence rate inferior to that of their centralized counterparts or rely on extra assumptions. It remains an open problem whether a decentralized variance reduction algorithm can approach the performance of its centralized counterpart. In addition to the convergence rate, prior linearly convergent decentralized stochastic methods are confined to the smooth optimization problem and the composite optimization problem where the objective function is the sum of a smooth function and a non-smooth but convex one, is substantially less explored. Decentralized proximal algorithms based on full gradients that converge linearly are proposed by \cite{alghunaim2019linearly, sun2019convergence}, and \cite{ye2020decentralized}. On the other hand, to the best of our knowledge, there is no decentralized stochastic algorithm with a linear convergence rate that can solve the composite optimization problem.

In this paper, we consider the decentralized composite optimization problem where agents connected by a network cooperatively minimize a sum of smooth functions, plus a non-smooth and convex one. In practical, the regularized empirical risk minimization problem with data stored across a network is naturally cast as a decentralized composite optimization problem. We propose a novel \textbf{Proximal-gradient} algorithmic framework, called \texttt{PMGT-VR}, which is based on a combination of \textbf{Multi-consensus, Gradient Tracking, and Variance Reduction}. Our study indicates that \texttt{PMGT-VR} methods can solve the composite optimization problem and can achieve a linear convergence rate that matches their centralized counterparts. In particular, two representative algorithms, \texttt{PMGT-SAGA, PMGT-LSVRG} are described and analyzed in detail. Our methodology can be extended to other variance reduction techniques in a similar fashion. 

\begin{figure}[!t]
\centering
\includegraphics[width=3in]{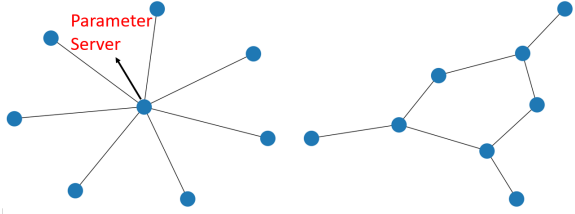}
\caption{The left figure is a master/slave network; the right figure is a decentralized network.}
\label{fig:topology}
\end{figure}

\subsection{Contributions}
We summarize our contributions as follows:
\begin{enumerate}
    \item We propose a novel algorithmic framework aiming at solving the decentralized composite convex problem over a connected and undirected network. A large family of decentralized variance reduction algorithms can be obtained in a similar fashion.
    \item We establish the linear convergence rate of \texttt{PMGT-SAGA} and \texttt{PMGT-LSVRG} for strongly convex composite optimization problems. The established rate matches that of centralized \texttt{SAGA} and \texttt{SVRG}. 
    \item Many existing works including proximal algorithms proposed by \cite{alghunaim2019linearly} and \cite{sun2019convergence} require that each agent's local loss function is convex. On the other hand, our methods could apply to the sum-of-nonconvex setting where the sum of local loss functions is strongly convex, but local functions can be nonconvex. 
    \item To the best of our knowledge, the proposed \texttt{PMGT-VR} is the first decentralized stochastic proximal method that can achieve a linear convergence rate. 
\end{enumerate}

\subsection{Related work}

We review existing works that are closely related to our framework in this subsection. 
\cite{nedic2009distributed} proposed decentralized subgradient method and its stochastic variant can be found in \cite{ram2010distributed}. 
In the distributed subgradient method, each agent performs
a consensus step and then a subgradient descent with a
diminishing step-size.
However, this kind of algorithms can not achieve a linear convergence even for smooth strongly convex problem \cite{yuan2016convergence}.
To conquer this problem, \texttt{EXTRA} proposed to track differences of gradients achieves the linear convergence rate \cite{shi2015extra}.
After \texttt{EXTRA}, a novel method named gradient tracking was proposed which also tracks differences of gradients  \cite{qu2017harnessing}.
Recently, several gradient-tracking based decentralized algorithms have been proposed to achieve fast convergence rates and efficient communication \cite{li2019nids,li2020revisiting,mudag2020ye,qu2017harnessing,qu2019accelerated,yuan2018exact}.

Decentralized composite optimization is another important research topic and there are many works focusing on it \cite{aybat2017distributed,hong2017prox,shi2015pgextra}. 
However, before the work of \cite{alghunaim2019linearly} and \cite{sun2019convergence}, all decentralized proximal algorithms can only achieve a sublinear convergence rate even though $f(x)$ is smooth strongly convex.
Recently, \cite{Xu2020Unified} and \cite{alghunaim2020decentralized} proposed  unified frameworks to analyze a large group of algorithms, and showed that gradient-tracking based decentralized algorithms can also achieve linear convergence rates with nonsmooth regularization
term such as \texttt{EXTRA} ( \texttt{PG-EXTRA} ) \cite{shi2015pgextra}, \texttt{NIDS} \cite{li2019nids}, and \texttt{Harnessing} \cite{qu2017harnessing}.
\cite{ye2020decentralized} proposed \texttt{DAPG} which is the first accelerated decentralized proximal gradient descent.
\texttt{DAPG} achieved the best-known computation and communication complexities.
Despite intensive studies in the literature, it is still hard to extend a smooth decentralized algorithm to non-smooth composite setting with the same convergence property.
For example, \cite{seaman2017optimal} proposed the communication optimal decentralized algorithm for smooth convex optimization.
However, it is still an open question to design a communication optimal decentralized algorithm for non-smooth composite convex optimization.

When each $f_i(x)$ has the finite-sum form, variance reduction  is an important and effective method in stochastic convex optimization \cite{saga2014Aaron,johnson2013accelerating,schmidt2017minimizing}.
To reduce the computational cost of decentralized optimization, \cite{mokhtari2016dsa} proposed the first decentralized variance reduction method named  \texttt{DSA} which  integrates \texttt{EXTRA} \cite{shi2015extra} with \texttt{SAGA} \cite{saga2014Aaron}.
\texttt{DBSA} tried to use proximal mapping to accelerate \texttt{DSA} \cite{shen2018towards}.
\texttt{AFDS} which combines variance reduction with acceleration technique can achieve the optimal communication complexity for variance reduction based algorithms \cite{hendrikx2019accelerated}.
However, \texttt{DBSA} and  \texttt{AFDS} relied on the assumption that the proximal mapping with respect to $f(x)$ can be solved efficiently.
To conquer the expensive cost of solving proximal mapping, \cite{xin2020variance} brought up \texttt{GT-SAGA} and \texttt{GT-SVRG} which only use the gradient of each individual function.
However, these two algorithms can only achieve much inferior computation and communication complexities  compared with \texttt{DBSA}.
\texttt{DVR} is another proximal mapping free algorithm which achieves better performance than \texttt{GT-SAGA} and \texttt{GT-SVRG} but still inferior to \texttt{DBSA} and  \texttt{AFDS} \cite{hendrikx2020dual}.
Recently, \cite{li2020optimal} proposed an algorithm which can achieve optimal computation and communication complexities. However, the algorithm required that the local loss function $f_j$'s are strongly convex and each $f_{i,j}$ is convex. When the local loss functions $f_j$'s are similar, \cite{gt_sarah2020shicong} demonstrated that gradient tracking and extra averaging are helpful for lower computation and communication complexities. Though many decentralized variance reduction algorithms have been proposed and achieved good performance, it is hard to generalize them to handle decentralized composite optimization. 

\subsection{Paper organization}
In section~\ref{sec:notation}, we introduce the formulation of the decentralized composite optimization problem over a network, notations used in this paper, and several important concepts. In section~\ref{sec:algodevelop}, we discuss the limitation of existing works, motivate our approach, and describe the \texttt{PMGT-VR} framework and two representative algorithms: \texttt{PMGT-SAGA} and \texttt{PMGT-LSVRG} in detail. In section~\ref{sec:result}, we provide the main results of this paper and compare the proposed algorithms with existing state-of-the-art works. In section~\ref{sec:analysis}, we prove the theoretical results of the proposed algorithms. We present numerical simulations in Section~\ref{sec:exp} and conclude in Section~\ref{sec:conclude}. 

\section{Problem Formulation and Preliminaries}\label{sec:notation}
\subsection{Problem Formulation}
We consider the following decentralized composite optimization problem:
\begin{equation}
    \min_{x\in\RR^d} h(x)\triangleq  f(x)+r(x),
\end{equation}
with $f(x) \triangleq \frac{1}{m}\sum_{i=1}^m f_i(x)$ and $f_i(x)\triangleq \frac{1}{n}\sum_{j=1}^{n} f_{i,j}(x)$ where $m$ is the number of agents whose local training sets are of equal size $n$. The function $f_i(x)$ is the local loss function private to agent $i$; $f_{i,j}(x)$ denotes the local loss function at the $j$-th training example of agent $i$; $r(x)$ is a non-smooth and convex regularization term shared by all agents. Note that it has been shown that with algorithms unlimited in the number of decentralized communication steps but limited to one gradient and proximal computations per iteration, it is not possible to achieve a linear convergence rate when agents have different regularization functions by \cite{alghunaim2020decentralized}. We will make the following assumptions for the rest of this paper:
\begin{assumption}\label{ass:lsmooth}
    Each $f_{i,j}:\mathbb{R}^d \to \mathbb{R}$ is $L$-smooth: 
    \begin{equation}
        f_{i,j}(y) - f_{i,j}(x) \leq \langle \nabla f_{i,j}(x),y-x \rangle + \frac{L}{2}\|y-x\|^2.
    \end{equation}
    This implies that 
	\begin{equation}
	\begin{aligned}
	\label{eq:str_cvx}
	&\frac{1}{2L}\norm{\nabla f_{i,j}(x) - \nabla f_{i,j}(y)}^2\\
	\le&
	f_{i,j}(x) - f_{i,j}(y) - \dotprod{\nabla f_{i,j}(y), x-y}, \forall x, y\in\RR^d.
	\end{aligned}
	\end{equation}
\end{assumption}

\begin{assumption}\label{ass:strongcon}
    $f:\mathbb{R}^d \to \mathbb{R}$ is $\mu$-strongly convex:
    \begin{equation}
        f(y) - f(x) \geq \langle \nabla f(x),y-x \rangle + \frac{\mu}{2}\|y-x\|^2,
    \end{equation}
\end{assumption}
From assumption~\ref{ass:lsmooth}, we know $f$ is also $L$-smooth. For any $L$-smooth and $\mu$-strongly convex function $f(x)$, we have $L \geq \mu$ and we define the condition number of $f(x)$ as $\kappa=\frac{L}{\mu}$. Due to assumption~\ref{ass:strongcon} and the convexity of $r(x)$, $h(x)$ has a unique global minimizer which is denoted as $x^*$. We also assume that the regularization function $r(x)$ is proximable in the sense that its proximal mapping 
$$
\proximal_{\eta,r}(x)=\argmin_{z\in\mathbb{R}^d}\Big(r(z)+\frac{1}{2\eta}\|z-x\|^2\Big),
$$
can be computed efficiently. For instance, the proximal mapping of $L_1$-regularization function is of a closed form. 

\textbf{Decentralized Communication} In the decentralized setting, due to the absence of the parameter server, the agents cooperatively minimize the function $h(x)$ based only on decentralized communications over a given network. The decentralized communication is defined through an undirected graph $\mathcal{G}$ of $m$ nodes where each node of the graph corresponds to an agent. In a decentralized communication step, each agent is allowed to send $\cO(1)$ vectors of size $d$ to their neighbors. 

Gossip algorithms are generally used in the decentralized setting \cite{shi2015extra, Xu2020Unified, gt_sarah2020shicong, xin2020variance}. To model a gossip communication step, we introduce the gossip matrix $W \in \RR^{m \times m}$. Agent $i$ and $j$ are allowed to exchange information if and only if there exists an edge between them in the graph $\mathcal{G}$, that is, $w_{ij} \neq 0$. In a gossip communication step, each agent $i$ will receive the current iterates of her neighbors. Then, each agent $i$ updates her iterate by a weighted average of the received iterates and her iterate, that is, $\sum_{j=1,w_{ij}\neq 0}^m w_{ij}\xb_j$. If we define the following aggregated notation of the column vectors $x_1,...,x_m$:
\begin{equation}
\xb 
=
[
x_1,\cdots,x_m]^\top,
\end{equation}
the decentralized communication can be abstracted as multiplication by the gossip matrix:
$$\xb^{\text{new}} = W\xb^{\text{old}}.$$
The gossip matrix is assumed to satisfy several conditions.
\begin{assumption}\label{ass:w} Let $W$ be the gossip matrix. We assume that 
\begin{enumerate}
    \item W is symmetric;
    \item $\mathbf{0} \preceq W \preceq I, W \mathbf{1}=\mathbf{1},\operatorname{null}(I-W)=\operatorname{span}(\mathbf{1})$,
    where $I$ is the $m \times m$ identity matrix and $\mathbf{1}$ is the $m$-dimensional all one column vector.
\end{enumerate}
\end{assumption}

As a result of assumption~\ref{ass:w}, the second-largest singular value of $W$, denoted as $\lambda_2(W)$, is strictly less than $1$. It is also called the mixing rate of the network topology since we have $\|W\xb-\frac{1}{m}\mathbf{11^\top}\xb\| \le \lambda_2(W) \norm{\xb - \frac{1}{m}\mathbf{11^\top}\xb}$. Therefore, $\lambda_2(W)$ indicates how fast the variables will be averaged through decentralized communications. For instance, a fully connected network with $W=\frac{1}{m}\mathbf{11^\top}$ has $\lambda_2(W)=0$ but each agent communicates from and to all other agents; exp$2$-ring has $\lambda_2(W)=1-2/(2\lfloor\log_2(m-1)\rfloor)$ and each agent is connected to $\lfloor\log_2(m-1)\rfloor$ neighbors. With the exp$2$-ring topology, each agent is connected to the agent $2^0,2^1,2^2,...$ hops away and the weights are set uniformly. Many existing decentralized optimization algorithms can be implemented efficiently with exp$2$-ring by \cite{bluefrog2021}.

\subsection{Notation}
Through the rest of this paper, $\norm{\cdot}$ denotes 2-norm for vectors and Frobenius norm for matrices. $\norm{\cdot}_1$ denotes $1$-norm for vectors. $\langle\cdot,\cdot\rangle$ is the inner product of two vectors. 
When notation "$\leq$" is applied to vectors of the same dimension, it means element-wise "less than or equal to". We also define the following notations:
\begin{equation}
\nabla F(\xb)
=
[
\nabla f(\xb_1),\cdots,\nabla f(\xb_m)]^\top,
\end{equation}
and
\begin{equation}
\bbx = \frac{1}{m}\mathbf{1}^\top \xb,
\;
R(\xb) = \frac{1}{m}  \sum_{i=1}^{m}r(\xb_i),
\end{equation}
where we use the convention that $\xb_i$ is the $i$-th row of matrix $\xb$. Besides, the iteration index will be used as superscript. For instance, $\xb_i^t$ is the iterate of agent $i$ at iteration $t$. In next subsection, We will introduce several important concepts and their related notations including $\bs, \vb, \bar{s}, \bar{v}$. These notations are defined similarly. Moreover, we denote the aggregated proximal operator as
\begin{equation}
\label{eq:proxes}
\proximal_{m\eta,R}(\xb)=\argmin_{\zb\in\mathbb{R}^{m\times d}}\Big(R(\zb)+\frac{1}{2m\eta}\|\zb-\xb\|^2\Big).
\end{equation}
We will also use the following notations. $D_f(x,y) = f(x)-f(y)-\langle \nabla f(y), x-y\rangle$ is the \textit{Bregman divergence} associated with $f(\cdot)$. Similar to the centralized analysis framework \cite{unifiedframework2019peter}, the definitions of \textit{gradient learning quantity} $\Delta^t$ and \textit{Lyapunov function} $V^t$ can vary depending on the associated variance reduction technique. For \texttt{PMGT-SAGA}, we have
\begin{equation}
    \label{eq:gradient_learning_dsaga}
    \Delta^t = \frac{1}{mn}\sum_{i,j=1}^{m,n} \| \nabla f_{i,j}(\phi_{i,j}^{t})-\nabla f_{i,j}(x^*)\|^2.
\end{equation}
For \texttt{PMGT-LSVRG}, we have
\begin{equation}
    \label{eq:gradient_learning_dlsvrg}
    \Delta^t = \frac{1}{mn}\sum_{i,j=1}^{m,n} \| \nabla f_{i,j}(\wb^t_i)-\nabla f_{i,j}(x^*)\|^2.
\end{equation}
For both \texttt{PMGT-SAGA} and \texttt{PMGT-LSVRG}, we have
\begin{equation}
    \label{eq:lyapunov}
    V^t=\|\bar{x}^t - x^*\|^2 + 4n\eta^2 \Delta^{t}.
\end{equation}
Variables $\eta, \phi_{i,j}^t$, and $\wb_t^i$ are defined in Algorithm~\ref{alg:dsaga} and Algorithm~\ref{alg:dlsvrg} and are introduced in section~\ref{sec:algodevelop}.

The performance of a decentralized optimization algorithm is usually measured by two quantities, namely, the computation complexity $T$ and the communication complexity $C$. The computational complexity and the communication complexity are the number of component gradient evaluations (evaluation of $\nabla f_{i,j}(x)$) and the number of decentralized communications for each agent, respectively, to achieve $\epsilon$-approximate solution which is defined as
\begin{equation}
    \label{eq:epsilon_optimal}
    \max\left\{\frac{1}{m}\norm{\xb^t-\mathbf{1}\bar{x}^t}^2, \norm{\bar{x}^t - x^*}^2 \right\}< \epsilon.
\end{equation}
in the decentralized setting. 

\subsection{Preliminaries}
The \texttt{PMGT-VR} framework is based on several techniques including variance reduction, gradient tracking, and multi-consensus. We briefly introduce these concepts in this subsection. Their impacts on the proposed decentralized framework will be illustrated in sections~\ref{sec:algodevelop} and~\ref{sec:analysis}.

\noindent\textbf{SGD and Variance Reduction} We consider a single-machine empirical risk minimization problem where the objective function is of the form: $g(x) := \frac{1}{n}\sum_{i=1}^n g_i(x)$ and we assume that each $g_i$ is $L$-smooth and $\mu$-strongly convex. Therefore, $g(x)$ has a unique minimizer $x^*$. Stochastic gradient descent (\texttt{SGD}) is a popular and powerful approach and has been extensively used for this problem when $n$ is large. \texttt{SGD} estimates the full gradient $\nabla g(x)$ by a stochastic one $\nabla g_i(x)$ where $i$ is picked from $\{1,2,..,n\}$ at random. Since computing $\nabla g_i(x)$ is roughly $n$ times faster than computing the full gradient $\nabla g(x)$, \texttt{SGD} is of a lower per-iteration cost as compared to gradient descent (\texttt{GD}) based on full gradients. However, \texttt{SGD} can only admit a sublinear convergence rate even for strongly convex and smooth problems due to the variance of gradient estimator $\nabla g_i(x)$. To reach an $\epsilon$-approximate solution, i.e., $\EE[\norm{x-x^*}^2] < \epsilon$, \texttt{SGD} requires $\cO(\frac{1}{\epsilon})$ times component gradient evaluation. Variance reduction techniques are designed for the stochastic optimization problem where a better gradient estimator $v$ is used so that the variance gradually tends to zero. Variance reduction methods have a low per-iteration cost similar to \texttt{SGD} and, at the same time, achieve a linear convergence rate. Many existing variance reduction methods require $O((n+\kappa)\cdot \log \frac{1}{\epsilon})$ times component gradient evaluation to achieve an $\epsilon$-approximate solution \cite{johnson2013accelerating,lsvrg2019kovalev,sarah2017Lam,saga2014Aaron}.

\noindent\textbf{Gradient tracking} Gradient tracking proposed by \cite{zhu2010gradientracking} and later adopted by \cite{qu2017harnessing} is a powerful gradient estimation scheme for the decentralized optimization problem. The innovative idea of gradient tracking is to asymptotically approach the global gradient $\nabla f(\xb_i^t)$ by a gradient tracker $\bs_i^t$ through local computations and decentralized communications. This gradient estimation scheme performs the following update at each iteration
$$\bs^{t+1} = W\bs^t+\vb^{t+1}-\vb^t,$$
where $\vb^{t}_i$ is the local (variance reduction) gradient estimator of $\nabla f_i(x)$ private to agent $i$ at iteration $t$. With the aid of gradient tracking, agents can dynamically track the average of local gradient estimators $\vb_i^t$ by the average of global gradient trackers $\bs^t_i$, that is,
$$\bar{s}^{t+1} = \bar{v}^{t+1}.$$
Under the assumptions~\ref{ass:lsmooth} and~\ref{ass:w}, we can show that as the consensus errors $\norm{\xb^t-\mathbf{1}\bar{x}^t}$ and $\norm{\bs^t-\mathbf{1}\bar{s}^t}$ tend to zero, one has $\bs_i^t \to \bar{s}^t \to \nabla f(\bar{x}^t)$.

\noindent\textbf{Multi-consensus}
Multi-consensus simply means that there are multiple decentralized communication steps within one iteration. If we denote the number of decentralized communication steps as $K$, multiplication by $W^K$ improves the mixing rate from $\lambda_2(W)$ to $\lambda_2(W)^{K}$. Therefore, if we set $K$ to infinity, all agents will share $\bar{x}^t$ which is the same as the master/slave case. Motivated by this observation, we choose $K>1$ so that the convergence rate will not be degraded by a poorly connected network (where $\lambda_2(W)$ is close to $1$). Moreover, multi-consensus can be naturally accelerated (e.g. via algorithm~\ref{alg:mix}, \texttt{FastMix}) and leads to a better dependence on the topology. 

There are two gradient estimators in the \texttt{PMGT-VR} framework (Algorithm \ref{alg:mgt}). Whereas the term "global gradient tracker" refers to the gradient-tracking estimator $\bs^t_i$, "local gradient estimator" refers to the local variance reduction estimator $\vb_i^t$. 

\begin{algorithm}[tb]
	\caption{PMGT-VR Framework}
	\label{alg:mgt}
	\begin{small}
		\begin{algorithmic}[1]
			\STATE {\bf Input:} $\xb_i^{0}=\xb_j^{0}$ for $1\le i,j,\le m$, $\vb^{-1}=\bs^{-1} = \nabla F(\xb^0)$,  $\eta = \frac{1}{12L}$ and  $K = \cO\left(\frac{1}{\sqrt{1-\lambda_2(W)}}\log \max \left(\kappa, n\right)\right)$.
			\FOR {$t=0,\dots, T$ }
 
			\STATE Update the local stochastic gradient estimators $\vb^{t}$;
			\STATE Update the local gradient trackers as 
			$\bs^{t} = \mathrm{FastMix}\left(\bs^{t-1} + \vb^{t} - \vb^{t-1}, K\right)$.
			\STATE  $\xb^{t+1} =\mathrm{FastMix}(\proximal_{\eta m,R}( \xb^{t} - \eta \bs^{t}),K) $; 
			\ENDFOR
			\STATE {\bf Output:} $\xb^{T+1}$.
		\end{algorithmic}
	\end{small}
\end{algorithm}

\section{Algorithm Development}\label{sec:algodevelop}
The complete \texttt{PMGT-VR} framework is described in algorithm~\ref{alg:mgt} which is based on a combination of variance reduction (line 3), gradient tracking (line 4), and multi-consensus (\texttt{FastMix}). In what follows, we start with \texttt{DSGD}, and discuss the challenges of decentralized stochastic optimization algorithm design. We will add these ingredients one by one to illustrate how they affect the performance of the proposed algorithmic framework. Let us assume that $r(x)=0$ first and we will extend to the composite optimization problem later.
\subsection{The General Framework}
\textbf{Decentralized SGD} \texttt{DSGD} is based on decentralized communication and updates with local gradients \cite{ram2010distributed, dsgd2017lian} where each agent $i$ performs the following iterative update:
\begin{equation*}
    \begin{aligned}
    \xb^{t+1}_i &= (W\xb^{t})_i - \nabla f_{i,j_i}(\xb_i^{t}),
    \end{aligned}
\end{equation*}
where $\nabla f_{i,j_i}(\xb_i^{t-1})$ is the local stochastic gradient. Under the assumption that each $f_i$ is $L$-smooth and $\mu$-strongly convex, \cite{yuan2019dsgd} shows that with a constant stepsize, $\EE \norm{\xb_i^t-x^*}^2$ decays linearly to a neighbourhood of the minimizer which is characterized as
\begin{align}
&\limsup _{t \rightarrow \infty} \frac{1}{n} \sum_{i=1}^{n} \EE\left[\left\|\xb_i^t-x^*\right\|_{2}^{2}\right]\notag\\
&=\mathcal{O}\left(\frac{\eta \sigma^{2}}{m \mu}+\frac{\eta^{2} \kappa^{2} \sigma^{2}}{1-\lambda_2(W)}+\frac{\eta^2\kappa^2\sum_{i=1}^m \norm{\nabla f_i(x^*)}^2}{m(1-\lambda_2(W))^2}\right), \label{eq:aa}
\end{align}
where $\sigma^2$ is the upper bound of variances 
of local gradient noise and $\eta$ is a constant stepsize. Unlike \texttt{SGD}, in addition to the variance of local stochastic gradients, the performance of \texttt{DSGD} is also degraded by the dissimilarity among the datasets across the agents.
We can see that \texttt{DSGD} has an additional bias (the third term of Eqn.~\eqref{eq:aa}) which could be arbitrarily large due to the dissimilarity among the datasets across agents. Moreover, the minimizer $x^*$ may not necessarily be a fixed point (in expectation) of the above update as $\nabla f_i(x^*)$ can be non-zero in general. 

\textbf{Gradient tracking} The above issues can be overcome by using the global full gradient $\nabla f(x)$. However, agents have no access to the global gradient and thus an efficient global gradient estimation scheme based on decentralized communication is required. Motivated by this observation, \cite{shi2018gtdsgd} proposed \texttt{GT-DSGD} whose iterative update is performed as
\begin{equation*}
    \begin{aligned}
    \xb^{t+1}_i &= (W\xb^{t})_i - \bs^{t}_i, \\
    \bs^{t+1}_i &= (W\bs^{t})_i + \nabla f_{i,j_{i}}(\xb_i^{t+1}) - \nabla f_{i,j_{i}}(\xb_i^{t}).
    \end{aligned}
\end{equation*}
With gradient tracking, each agent $i$ can approach $\nabla f(\xb_i^t)$ by the global gradient tracker $\bs_i^t$ and this approximation eliminates the bias in the theoretical bound of \texttt{DSGD}. In this case, when the network is well connected, the convergence behavior of \texttt{GT-DSGD} is determined only by the stepsize sequence and the variance of local stochastic gradient which is similar to \texttt{SGD}.

\textbf{Variance reduction} \texttt{GT-DSGD} can only achieve a sublinear convergence rate. Similar with the centralized case, it is natural to use variance reduction gradient estimator in \texttt{GT-DSGD} to estimate the local full gradient $\nabla f_i(x)$. \cite{xin2020variance} explored this idea and proposed \texttt{GT-SVRG} and \texttt{GT-SAGA} that converge linearly under assumptions $1-3$. The computation complexities (the same as communication complexities) of \texttt{GT-SVRG} and \texttt{GT-SAGA} to reach an $\epsilon$-approximate solution are $\cO\left((n+\frac{\kappa^2\log\kappa}{(1-\lambda_2(W))^2})\log\frac{1}{\epsilon}\right)$ and $\cO\left((n+\frac{\kappa^2}{(1-\lambda_2(W))^2})\log\frac{1}{\epsilon}\right)$, respectively. We can see that a \texttt{GT-VR} algorithm can only achieve a convergence rate inferior to that of its centralized counterpart. In particular, the computation complexity is worse by a factor dependent on the network topology which is the cost of being decentralized.

\textbf{Multi-consensus} 
In the decentralized setting, agents asymptotically achieve consensus through decentralized communications. Combining with gradient tracking, both consensus errors and gradient tracking error (formally defined in section~\ref{sec:analysis}) asymptotically tend to zero in the gradient tracking framework. In this case, we may expect that the decentralized algorithms asymptotically behave similarly with their centralized counterparts. However, the consensus steps will be influenced by the mixing rate $\lambda_2(W)$. From the theoretical guarantees of \texttt{GT-VR} and \texttt{GT-DSGD}, we can see that the performance of an algorithm with gradient tracking degrades when the mixing rate is close to $1$. This issue can be overcome by involving $K$ decentralized communication steps within one iteration to improve the mixing rate from $\lambda_2(W)$ to $\lambda_2(W)^K$. This motivates the usage of multi-consensus. With $K$ suggested in Theorem~\ref{th:convergence_rate}, the mixing rate is improved to be good enough so that the iteration complexity will be independent of the network topology. In this case, we can approach the centralized performance even with a poorly connected network. Moreover, multi-consensus also contributes to a better dependence on the topology because it can be naturally accelerated via algorithm~\ref{alg:mix}. The resulting faster convergence rate and accelerated consensus steps compensate the extra communication costs of multi-consensus. 
Consequently, the overall communication complexity of \texttt{PMGT-VR} is also generally lower as compared to \texttt{GT-VR} algorithms.


To summary, the introduced gradient tracking and multi-consensus essentially make the proposed algorithms asymptotically behave like a centralized one independent of the network topology. Therefore, the instantiations of \texttt{PMGT-VR} can approximate their centralized counterparts and have similar convergence properties. In particular, since the proposed framework uses variance reduction technique, it converges linearly to the minimizer $x^*$ under assumptions $1-3$. 


\textbf{Extension to the composite setting} Finally, one can further incorporate the proximal mapping with our framework. Although it is usually non-trivial to extend a smooth decentralized algorithm to the composite setting with the same convergence rate, the extension is rather natural in our framework. Our analysis relies on characterizing the interrelationship of several quantities (formally defined in section~\ref{sec:analysis}) by a linear system inequality. In our framework, the usage of proximal mapping results in another inequality about the considered quantities resulting from the non-expansiveness of the proximal mapping. With or without the proximal mapping, the obtained linear system inequalities are different in the coefficients. However, we control the spectral radius of the coefficient matrices directly by the parameter $K$ so the techniques for the smooth case can be extended to the composite case but with different per-iteration number of communications.


\subsection{PMGT-SAGA and PMGT-LSVRG}
Different choices of local variance reduction gradient estimator $\vb$ (line $4$ of algorithm~\ref{alg:mgt}) lead to different \texttt{PMGT-VR} algorithms. In this section, we describe and compare two representative algorithms \texttt{PMGT-SAGA} and \texttt{PMGT-LSVRG} that are described in algorithm~\ref{alg:dsaga} and algorithm~\ref{alg:dlsvrg}, respectively.

\begin{algorithm}[H]
	\caption{PMGT-SAGA}
	\label{alg:dsaga}
	\begin{small}
		\begin{algorithmic}[1]
			\STATE {\bf Input:} $\xb^0_{i}=\xb^0_{j}$ for $1\le i,j,\le m$, $\vb^{-1}=\bs^{-1} =  \nabla F(\xb^0)$,  $\eta = \frac{1}{12L}$ and  $K = \frac{1}{\sqrt{1 - \lambda_2(W)}} \log 41\max\left(24 \kappa, 4 n\right)$.
			\STATE Take $\phi_{i,j}^{0} = \xb_i^{0}, \forall i \in \{1,2,...,m\}, j \in \{1,2,...,n\}$.
			\FOR {$t=0,\dots, T$ }
			\STATE In parallel, for each agent $i$:
			\STATE Pick a $j_i$ uniformly at random from $\{1,\dots, n\}$.
			\STATE
			Take $\phi_{i,j_i}^{t+1} = \xb_i^{t}$.
			\STATE Update the local variance reduction estimator:
			\begin{equation*}
			\vb_i^{t} = \nabla f_{i,j_i}(\phi_{i,j_i}^{t+1}) - \nabla f_{i,j_i}(\phi_{i,j_i}^{t}) + \frac{1}{n}\sum_{j=1}^{n} \nabla f_{i,j}(\phi_{i,j}^{t}).
			\end{equation*}
			\STATE Update the local gradient tracker as 
			$\bs^{t} = \mathrm{FastMix}\left(\bs^{t-1} + \vb^{t} - \vb^{t-1}, K\right)$.
			\STATE  $\xb^{t+1}_i =\mathrm{FastMix}(\proximal_{\eta m,R}(\xb^{t} - \eta \bs^{t}),K)_i $. 
			\ENDFOR
			\STATE {\bf Output:} $\xb^{T+1}$.
		\end{algorithmic}
	\end{small}
\end{algorithm}

\texttt{PMGT-SAGA} is a \texttt{SAGA}-based implementation of the proposed framework. As the original \texttt{SAGA}, each agent needs to store a table of gradients for the local variance reduction gradient estimator. At each iteration $t$, each agent $i$ picks $j_i$ from $\{1,2,...,n\}$ at random and it replaces $\nabla f_{i,j_i}(\phi_{i,j_i}^{t+1})$ with $\nabla f_{i,j_i}(\xb_i^{t})$ and all other entries of the table remain unchanged. Then, each agent $i$ updates her local gradient estimator by
\begin{equation}
	\label{eq:phi}
\vb_i^{t} = \nabla f_{i,j_i}(\phi_{i,j_i}^{t+1}) - \nabla f_{i,j_i}(\phi_{i,j_i}^{t}) + \frac{1}{n}\sum_{j=1}^{n} \nabla f_{i,j}(\phi_{i,j}^{t}),
\end{equation}
where $\phi^{t}_{i,j}$ is the most recent iterate before iteration $t$ at which $\nabla f_{i,j}(\cdot)$ is evaluated. We see that \texttt{PMGT-SAGA} has a storage complexity of $\mathcal{O}(dn)$ for each agent due to the gradient tables.

\begin{algorithm}[H]
	\caption{PMGT-LSVRG}
	\label{alg:dlsvrg}
	\begin{small}
		\begin{algorithmic}[1]
			\STATE {\bf Input:} $\xb^0_{i}=\xb^0_{j}$ for $1\le i,j,\le m$, $\vb^{-1}=\bs^{-1} = \nabla F(\xb^0)$, $\eta = \frac{1}{12L}$, $p=\frac{1}{n}$, and  $K = \frac{1}{\sqrt{1 - \lambda_2(W)}} \log 41\max\left(24 \kappa, 4 n\right)$.
			\STATE Take $\wb_i^0=\xb_i^0, \forall i \in \{1,2,...,m\}$.
			\FOR {$t=0,\dots, T$ }
			\STATE In parallel, for each agent $i$:
			\STATE Pick a $j_i$ uniformly at random from $\{1,\dots, n\}$.
			\STATE Update the local variance reduction estimator:
			\begin{equation*}
            \vb_i^{t} = \nabla f_{i,j_i}(\xb^{t}_i) - \nabla f_{i,j_i} (\wb^{t}_i)+\nabla f_i(\wb^{t}_i).
			\end{equation*}
			\STATE Take $\wb^{t+1}_i = \xb^{t}_i$ with probability $p$; otherwise $\wb_i^{t+1}=\wb_i^{t}$.
			\STATE Update the local gradient tracker as 
			$\bs^{t} = \mathrm{FastMix}\left(\bs^{t-1} + \vb^{t} - \vb^{t-1}, K\right)$.
			\STATE  $\xb^{t+1}_i =\mathrm{FastMix}(\proximal_{\eta m,R}(\xb^{t} - \eta \bs^{t}),K)_i $. 
			\ENDFOR
			\STATE {\bf Output:} $\xb^{T+1}$.
		\end{algorithmic}
	\end{small}
\end{algorithm}

\texttt{PMGT-LSVRG} is a \texttt{LSVRG}-based implementation of the proposed framework. \texttt{SVRG} proposed by \cite{johnson2013accelerating} contains two loops and computes the full gradient at the beginning of each inner loop. However, the theoretically optimal inner loop size depends on both $L$ and $\mu$ which may be hard to estimate for a real-world dataset and the analysis is also more complicated for handling the double loop structure. To conquer this dilemma, \cite{lsvrg2019kovalev} designs a variant of the original \texttt{SVRG}, called Loopless SVRG (\texttt{LSVRG}) in which the outer loop is removed and the agent updates the stored full gradient in a stochastic manner. In the case of \texttt{PMGT-LSVRG}, the local variance reduction gradient estimator is given by
\begin{equation}
	\label{eq:w_update}
    \vb_i^{t} = \nabla f_{i,j_i}(\xb^{t}_i) - \nabla f_{i,j_i} (\wb^{t}_i)+\nabla f_i(\wb^{t}_i),
\end{equation}
where $\wb^{t}_i$ is the most recent iterate at which $\nabla f_i(\cdot)$ is evaluated. 
At each iteration, each agent $i$ picks $j_i$ from $\{1,2,...,n\}$ at random and the reference point $\wb_i^{t+1}$ is replaced with a small probability $p$ by the $\xb_i^{t}$ and is left unchanged with probability $1-p$. For \texttt{LSVRG}, a simple choice of $p=1/n$ leads to a convergence rate identical to that of \texttt{SVRG}. 

By using the estimator in \texttt{PMGT-LSVRG}, agents need not store a table of gradients. However, the expected per-iteration number of gradient evaluations for each agent is $p \cdot (n+2)+(1-p)\cdot 2 = 3$ when $p=1/n$ for \texttt{PMGT-LSVRG}. Therefore, \texttt{PMGT-LSVRG} suffers from a higher expected per-iteration cost as compared to \texttt{PMGT-SAGA}. We observe that there is a trade-off between storage and computation efficiency here and users can implement \texttt{PMGT-VR} algorithms based on their customized needs.

\section{Main Results} \label{sec:result}
In this section, we establish the convergence rate of \texttt{PMGT-SAGA} and \texttt{PMGT-LSVRG} which matches that of centralized \texttt{SAGA} and \texttt{LSVRG}. Before continuing, we first define the vector of consensus errors.
\begin{equation}\label{def:consensus}
         \zb^t = [\frac{1}{m}\norm{\xb^t - \mathbf{1}\bbx^t}^2, \frac{\eta^2}{m} \norm{\bs^t - \mathbf{1}\bbs^t}^2]^\top.
\end{equation}

\subsection{Main Results}
\label{subsec:theory}
\begin{theorem}~{}
    \label{th:convergence_rate}
    Let assumptions~\ref{ass:lsmooth}, \ref{ass:strongcon}, and~\ref{ass:w} hold.
	If we set $p=\frac{1}{n}$ for \texttt{PMGT-LSVRG} and choose $K = \frac{1}{\sqrt{1 - \lambda_2(W)}} \log \frac{1}{\rho}$ where $\rho$ satisfies 
	\begin{equation}
	\label{eq:rho_cond}
    \rho \le \frac{1}{41}\min \left(\frac{1}{24 \kappa}, \frac{1}{4 n}\right),\\
	\end{equation}
	 and choose stepsize $\eta = 1/(12L)$, then, for both \texttt{PMGT-SAGA} and \texttt{PMGT-LSVRG} 
	 it holds that
	 \begin{equation}
	\label{eq:V_conv}
	\begin{aligned}
	\mathbb{E}\left[V^{t}\right] \leq& \max \left(1-\frac{1}{24 \kappa}, 1-\frac{1}{4 n}\right)^{t}\left( V^{0}+\left\|\mathbf{z}^{0}\right\|\right)\\
	\end{aligned}
	\end{equation}
	and
	\begin{equation}
	\label{eq:Z_conv}
	\begin{aligned}
	\mathbb{E}&\left[\max \left(\frac{1}{m}\left\|\mathbf{x}^{t}-\mathbf{1} \bar{x}^{t}\right\|^{2}, \frac{1}{144 m L^{2}}\left\|\mathbf{s}^{t}-\mathbf{1} \bar{s}^{t}\right\|^{2}\right)\right] 
    \leq \\
    & \left(\frac{1}{45387} \min \left(\frac{1}{36 \kappa^{2}}, \frac{1}{n^{2}}\right)+2^{-t}\right)\\
    &\cdot\max \left(1-\frac{1}{24 \kappa}, 1-\frac{1}{4 n}\right)^{t} \cdot\left(V^{0}+\left\|\mathbf{z}^{0}\right\|\right).
	\end{aligned}
	\end{equation}
	To find an $\epsilon$-approximate solution, the computation complexity $T$ and communication complexity $C$ are
	\begin{equation}
	\label{eq:complexity}
    \begin{aligned}
    T &= \cO\left(\max\left(\kappa, n\right) \log\frac{1}{\epsilon}\right),\\
    C &= \cO\left(\max\left(n\log n, \kappa \log \kappa\right)\frac{1}{\sqrt{1-\lambda_2(W)}}\log\frac{1}{\epsilon}\right).
    \end{aligned}
	\end{equation}	
\end{theorem}

\begin{remark}
\texttt{PMGT-SAGA} and \texttt{PMGT-LSVRG} admit a computation complexity (the same as iteration complexity) that matches that of their centralized counterparts.
\end{remark}

\begin{remark}
The communication complexity of the proposed algorithms implicitly depends on the number of agents $m$ through the second-largest singular value of the gossip matrix $W$. One may use the exp$2$-ring topology introduced in section~\ref{sec:notation} when the network can be designed. In this case, we have $\lambda_2(W)= 1-\frac{2}{2+\lfloor\log_2(m-1)\rfloor}$ and each agent communicates with $\lfloor\log_2(m-1)\rfloor)$ neighbors. Therefore, the per-iteration number of communications with neighbors for the proposed algorithms is $\cO\left(\left(\log \kappa + \log n\right)\log^{3/2} m\right)$. On the other hand, the per-iteration communication cost for the master/slave setting is $\cO(m)$. We can see that \texttt{PMGT-VR} methods can be preferable when $m$ is large as compared to their centralized counterparts.
\end{remark}

\begin{table*}
	\small
	\vspace{0.15in}
	\begin{center}
		\begin{tabular}{|c|c|c|c|}
			\hline
			Methods  & Problem & Complexity of computation & Complexity of communication\\
            \hline \texttt{ GT-SVRG } \cite{xin2020variance} & $f$ & $\mathcal{O}\left((n+\frac{\kappa^2\log\kappa}{(1-\lambda_2(W))^2})\log\frac{1}{\epsilon}\right)$ & $\mathcal{O}\left((n+\frac{\kappa^2\log\kappa}{(1-\lambda_2(W))^2})\log\frac{1}{\epsilon}\right)$\\
            
            \hline \texttt{ GT-SAGA } \cite{xin2020variance} & $f$ & $\mathcal{O}\left((n+\frac{\kappa^2}{(1-\lambda_2(W))^2})\log\frac{1}{\epsilon}\right)$ & $\mathcal{O}\left((n+\frac{\kappa^2}{(1-\lambda_2(W))^2})\log\frac{1}{\epsilon}\right)$ \\
            
            \hline \texttt{ PG-EXTRA } \cite{shi2015pgextra, Xu2020Unified}& $f+r$ & $\mathcal{O}\left(\frac{n\kappa}{(1-\lambda_2(W))}\log\frac{1}{\epsilon}\right)$ & $\mathcal{O}\left(\frac{n\kappa}{(1-\lambda_2(W))}\log\frac{1}{\epsilon}\right)$ \\
            
            \hline \texttt{ NIDS } \cite{li2019nids, Xu2020Unified}& $f+r$ & $\mathcal{O}\left(n(\kappa + \frac{1}{(1-\lambda_2(W))})\log\frac{1}{\epsilon}\right)$ & $\mathcal{O}\left((\kappa + \frac{1}{(1-\lambda_2(W))})\log\frac{1}{\epsilon}\right)$\\
            
            \hline \texttt{ Our methods } & $f+r$& $\mathcal{O}\left((n+\kappa)\log\frac{1}{\epsilon}\right)$ & $\mathcal{O}\left(\frac{(n\log n+\kappa \log \kappa)}{\sqrt{1-\lambda_2(W)}}\log\frac{1}{\epsilon}\right)$\\
			\hline
		\end{tabular}
	\end{center}
		\caption{Complexity comparisons between \texttt{PMGT-VR} algorithms and existing works for strongly convex problem. Note that \texttt{GT-SAGA} and \texttt{GT-SVRG} can only solve the smooth problems and other algorithms can solve both the smooth and composite problems. }
	\label{tbl:comp}
\end{table*}
\subsection{Comparison to existing algorithms}\label{subsec:comp}
In this subsection, we discuss the convergence properties established above. Table \ref{tbl:comp} presents a detailed comparison of our algorithms with state-of-the-art proximal algorithms (\texttt{PG-EXTRA} and \texttt{NIDS}) and the algorithms that are closely related to our framework (\texttt{GT-SAGA} and \texttt{GT-SVRG}).

We first note that the computation complexity 
of \texttt{PMGT-SAGA} and \texttt{PMGT-LSVRG} is the same as that of centralized \texttt{SAGA} and \texttt{LSVRG} which is not surprising as our algorithms are based on the imitation of their centralized counterparts and with a network improved by multi-consensus. 

We can see that the computation complexities of \texttt{GT-SAGA} and \texttt{GT-SVRG} are significantly worse than that of \texttt{PMGT-SAGA} and \texttt{PMGT-LSVRG}. Besides, as the faster convergence rate compensates for the extra rounds of decentralized communications in the proposed framework, the communication complexity of \texttt{PMGT-SAGA} and \texttt{PMGT-LSVRG} is also generally much lower than those of \texttt{GT-SAGA} and \texttt{GT-SVRG}. This illustrates the benefits of multi-consensus.

We only compare our algorithms with \texttt{NIDS} as \texttt{PG-EXTRA} is inferior to \texttt{NIDS}. We see that \texttt{PMGT-SAGA} and \texttt{PMGT-LSVRG} improve the computation complexity of \texttt{NIDS} from $\mathcal{O}\left(n(\kappa + \frac{1}{1-\lambda_2(W)})\log\frac{1}{\epsilon}\right)$ to $\mathcal{O}\left((n+\kappa)\log\frac{1}{\epsilon}\right)$ which illustrates the benefits of stochastic variance reduction algorithms over algorithms based on full gradients. On the other hand, the communication complexity of \texttt{PMGT-SAGA} and \texttt{PMGT-LSVRG} is $\mathcal{O}\left(\frac{(n\log n+\kappa \log \kappa)}{\sqrt{1-\lambda_2(W)}}\log\frac{1}{\epsilon}\right)$ which is in general worse than the $\mathcal{O}\left((\kappa + \frac{1}{1-\lambda_2(W)})\log\frac{1}{\epsilon}\right)$ communication complexity of \texttt{NIDS}. This can be interpreted as a trade-off between computation and communication efficiency. To further discuss this trade-off, we introduce the notation $\tau$ borrowed from \cite{hendrikx2020dual}, which is the relative ratio of communication cost and computation cost. In other words, evaluating $\nabla f_{i,j}(x)$ is of cost $1$ and the cost of one round of decentralized communication is $\tau$. The benefits of stochastic \texttt{PMGT-VR} methods are clearer for a relatively small $\tau$ which corresponds to the case when computation cost dominates. However, in numerical experiments section, we show that \texttt{PMGT-SAGA} and \texttt{PMGT-LSVRG} outperform \texttt{PG-EXTRA} and \texttt{NIDS} in terms of cost for a wide range of $\tau$.

\section{Convergence Analysis}
\label{sec:analysis}
\subsection{Convergence analysis: A sketch}
In this section, we provide several useful lemmas which illustrate the impacts of the ingredients used in the proposed framework. Throughout our analysis, we will consider the following quantities:
\begin{enumerate}
    \item consensus errors: $\frac{1}{m}\norm{\xb^t - \mathbf{1}\bbx^t}^2,\frac{\eta^2}{m} \norm{\bs^t - \mathbf{1}\bbs^t}^2$,
    \item the gradient tracking error: $\norm{\bbs^t - \nabla f(\bbx^t)}$,
    \item the gradient learning error: $\norm{\bbs_t - \nabla f(x^*)}^2$ and gradient learning quantity: $\Delta^t$ (defined in  Eqn.~\eqref{eq:gradient_learning_dsaga} and  Eqn.~\eqref{eq:gradient_learning_dlsvrg}),
    \item the convergence error: $\norm{\bbx^t-x^*}$ and Lyapunov function $V^t$ (defined in Eqn.~\eqref{eq:lyapunov}).
\end{enumerate} 
To summary, we will construct a linear system inequality about the above quantities. The idea is that the spectral radius of the coefficient matrix associated with this linear system inequality is less than $1$ if we use $K$ suggested in Theorem~\ref{th:convergence_rate}. In this case, both consensus errors and convergence error will decay linearly. Moreover, since our algorithms can approximate their centralized counterparts well, we can use similar techniques from the single-machine framework in \cite{unifiedframework2019peter}. 

\subsection{The linear system inequality}
To derive the desired linear system inequality, we show that the quantities are interrelated by the results of gradient tracking, proximal mapping and multi-consensus (\texttt{FastMix}). We start with the results of gradient tracking.
\begin{lemma}\label{lem:gradient_tracking_mgt}
	Let assumption~\ref{ass:lsmooth} hold. For both \texttt{PMGT-SAGA} and \texttt{PMGT-LSVRG}, it holds that $\bbs^t = \bar{v}^t$ and $\EE[\bbs^t] = \frac{1}{m}\sum_i^m \nabla f_i(\xb_i^t)$. Furthermore, we have 
	\begin{equation}
	\label{eq:var_s}
	\norm{\nabla f(\bbx^t) - \EE[\bbs^t]} 
	\le
	\frac{L}{\sqrt{m}}\norm{\xb^t - \mathbf{1}\bbx^t}.
	\end{equation}
\end{lemma}
We can see that $\bar{s}^t$ dynamically tracks the average of local variance reduction gradient estimators and the gradient tracking error is upper bounded by the consensus error. Then, we introduce several properties of proximal mapping and multi-consensus which result in an iterative inequality of consensus errors.

\begin{lemma}
    \label{lem:fastmix}
Let $\xb^0, \xb^K \in \RR^{m \times d}$ be the input and output of \texttt{FastMix} (Algorithm~\ref{alg:mix} in the appendix), respectively and $\bar{x}=\frac{1}{m}\mathbf{1}^\top\xb^0$. Then we have
    \begin{equation}
        \label{eq:fastmix}
        \begin{aligned}
        \norm{\xb^K-\mathbf{1}\overline{x}} \leq \rho \norm{\xb^0-\mathbf{1}\overline{x}}, \bar{x} = \frac{1}{m}\mathbf{1}^\top\xb^K.
        \end{aligned}
    \end{equation}
where $\rho=(1-\sqrt{1-\lambda_2(W)})^K$.
\end{lemma}

\begin{lemma}\label{lem:prox}
    Let $\proximal_{\eta m,R}^{(i)}(\xb)$ and $\xb_{i}$ denote the $i$-th row of the matrix $\proximal_{\eta m,R}(\xb)$ and $\xb$, respectively. Then, we have the following equality and inequality
	\begin{equation}\label{eq:local proximal equality}
	\proximal_{\eta m,R}^{(i)}(\xb)=\proximal_{\eta,r}(\xb_i),
	\end{equation}
	\begin{equation}
	\label{eq:prox_diff}
	\begin{aligned}
	&\norm{\proximal_{\eta  m,R}(\frac{1}{m}\mathbf{1}\mathbf{1}^\top \xb)-\frac{1}{m}\mathbf{1}\mathbf{1}^\top\proximal_{\eta  m,R}( \xb)}\\
	&\le
	\norm{  \xb-\frac{1}{m}\mathbf{1}\mathbf{1}^\top \xb }.
	\end{aligned}
	\end{equation}
\end{lemma}
The iterative inequality of consensus errors is stated in the next lemma:
\begin{lemma}
    \label{lem:consensus_error_mgt}
    For the general \texttt{PMGT-VR} framework, the consensus errors satisfy:
    \begin{equation}
    \label{eq:xxx}
        \begin{aligned}
    \frac{1}{m}\norm{\xb^{t+1} - \mathbf{1}\bbx^{t+1}}^2 \le&
	8\rho^2 \frac{1}{m}\norm{\xb^{t}-\mathbf{1}\bbx^{t}}^2\\
	&+
	8\rho^2 \frac{\eta^2}{m} \norm{\bs^t-\mathbf{1}\bbs^{t}}^2,\\
	\frac{\eta^2}{m} \norm{\bs^{t+1} - \mathbf{1}\bbs^{t+1}}^2
	\le&
	2\rho^2\frac{\eta^2}{m}\norm{\bs^t - \mathbf{1} \bbs^t}^2\\
	&+
	2\rho^2\frac{\eta^2}{m} \norm{\vb^{t+1} - \vb^t}^2,
        \end{aligned}
    \end{equation}
    where $\eta$ is a constant stepsize and $\rho=(1-\sqrt{1-\lambda_2(W)})^K.$
\end{lemma}

The linear system inequality is obtained once we can bound the $\norm{\vb^{t+1} - \vb^t}^2$ for a specific variance reduction gradient estimator. In what follows, the definitions of $\vb^t$ and the gradient learning quantity $\Delta^t$ depend on the algorithms we are considering.

\begin{lemma} \label{lem:local_gradient_estimator}
     For both \texttt{PMGT-SAGA} and \texttt{PMGT-LSVRG}, it holds that
    \begin{equation}
    	\begin{aligned}
    	&\EE \left[\frac{1}{m}\norm{\vb^{t+1} - \vb^t}\right]\\
    	\le& \left(8\rho^2 + 1\right)\frac{8L^2}{m}\norm{\xb^t - \mathbf{1}\bbx^t}^2+
		\frac{64\rho^2\eta^2 L^2}{m}\norm{\bs^t - \mathbf{1}\bbs^t}^2
		\\
		&+
		8L^2\norm{\bbx^{t+1} -x^*}^2 +4\Delta^{t+1}
		+
		8L^2\norm{\bbx^t -x^*}^2 + 4\Delta^t.
    	\end{aligned}
    \end{equation}
\end{lemma}

Substituting $\norm{\vb^{t+1}-\vb^t}$ in lemma \ref{lem:consensus_error_mgt}, the desired linear system inequality is stated in the next lemma.

\begin{lemma}
	\label{lem:consensus_error}
	Using the definition of $\zb^t$ in Eqn.~\eqref{def:consensus}, then for both \texttt{PMGT-SAGA} and \texttt{PMGT-LSVRG}, we have
	\begin{equation}
	\label{eq:Mat}
	\begin{aligned}
	&\EE
	\left[\zb^{t+1}\right]\\
	\le&
	2\rho^2\cdot
	\bigg(
	\begin{bmatrix}
		4, & 4\\
	8(8\rho^2 + 1)L^2\eta^2, & 64\rho^2\eta^2L^2 + 1
	\end{bmatrix}
	\cdot
	\zb^t
	\\&+\eta^2
	\begin{bmatrix}
	0\\
	8L^2(\norm{\bbx^{t+1} -x^*}^2+\norm{\bbx^t - x^*}^2) +4(\Delta^{t+1}+\Delta^t)
	\end{bmatrix}
	\bigg).
	\end{aligned}
	\end{equation}
\end{lemma}

We can see that with a small enough $\rho$ (corresponding to a large enough $K$), the Frobenius norm of the coefficient matrix is small than $1$ and so is the spectral radius.

\subsection{A unified analysis framework}\label{section:centralized_framework}
Following the standard analysis framework of single-machine algorithms based on \texttt{SGD} in \cite{unifiedframework2019peter}, we state the following two lemmas.

\begin{lemma}\label{lem:gradient_convergence}
	For both \texttt{PMGT-SAGA} and \texttt{PMGT-LSVRG}, it holds that
	\begin{equation}
	\label{eq:Var_}
	\begin{aligned}
	&\EE\left[\norm{\bbs_t - \nabla f(x^*)}^2\right]\\
	\le&
	8L \cdot D_f(\bbx^t, x^*)
	+
	2\Delta^t
	+
	\frac{4L^2}{m}\norm{\xb^t - \mathbf{1}\bbx^t}^2.
	\end{aligned}
	\end{equation}
\end{lemma}

\begin{lemma}\label{lem:gradient_learning_quantity}
	For \texttt{PMGT-SAGA} and \texttt{PMGT-LSVRG} with $p=\frac{1}{n}$, it holds that
	\begin{equation}
	\label{eq:Delta}
	\begin{aligned}
	&\EE\left[\Delta^{t+1}\right]\\
	\le&
	\left(1 - \frac{1}{n}\right)\Delta^t 
	+
	\frac{4L}{n} D_f(\bbx^t, x^*) 
	+
	\frac{2L^2}{mn}\norm{\xb^t - \mathbf{1}\bbx^t}^2.
	\end{aligned}
	\end{equation}
\end{lemma}

With the above lemmas in hand, we are ready to prove the following result involving Lyapunov function.
\begin{lemma}\label{lem:V_dec}
    Setting $\eta = \frac{1}{12L}$ and setting $p=\frac{1}{n}$ for \texttt{PMGT-LSVRG}, for both \texttt{PMGT-SAGA} and \texttt{PMGT-LSVRG}, it holds that
	\begin{equation}
	\label{eq:V_dec}
	\begin{aligned}
	&\EE\left[V^{t+1}\right]\\
	\le&
	\max\left(1-\frac{1}{12\kappa}, 1 - \frac{1}{2n}\right)\cdot V^t
	\\
	&+\frac{3\sqrt{V^t}}{\sqrt{m}}(\norm{\xb^t - \mathbf{1}\bbx^t}+\eta \norm{\bs^t-\mathbf{1}\bar{\bs}^t})
	\\    
	&+\frac{3.3}{m}\left\|\mathbf{x}^{t}-\mathbf{1} \bar{x}^{t}\right\|^{2}+\frac{2.5\eta^2}{m}\|\mathbf{s}^t-\mathbf{1}\bar{\mathbf{s}}^t\|^2.
	\end{aligned}
	\end{equation}
\end{lemma}

\subsection{Proof of Theorem \ref{th:convergence_rate}}
We now invoke the above results to give a detailed proof of Theorem \ref{th:convergence_rate}.
\begin{proof}[Proof of Theorem \ref{th:convergence_rate}] 

We will establish the convergence rate by induction. Since for $t=0$, each agent shares the same $\xb^0_i$ and $\bs^0_i$, we have $\norm{\xb^0-\mathbf{1}\bbx^0} = 0$ and $\norm{\bs^0-\mathbf{1}\bbs^0} = 0$. Therefore, by Eqn.~\eqref{eq:V_dec}, the following inequality holds for $t=1$:  
	\begin{equation}
		\label{eq:V_ass}
		\EE\left[V^{t}\right] 
		\le
		\left(\underbrace{\max\left(1 - \frac{1}{24\kappa}, 1 - \frac{1}{4n}\right)}_{\alpha}\right)^{t} \left(V^0+\norm{\zb^0}\right). 
	\end{equation}
	Now, we assume that the inequality holds when $t \le k$ and are going to prove Eqn.~\eqref{eq:V_ass} holds for $t=k+1$.
	Let us denote 
	\begin{equation}
	\Ab = \begin{bmatrix}
	4, & 4,\\
8(8\rho^2 + 1)L^2\eta^2, & 64\rho^2\eta^2L^2 + 1
	\end{bmatrix}.
	\end{equation}
	By the step size $\eta = \frac{1}{12L}$ and assumption that $\rho \le \frac{1}{41}\min \left(\frac{1}{24 \kappa}, \frac{1}{4 n}\right)$, we can obtain that
	\begin{align}
	    \label{eq:anorm}
		\norm{\Ab} < 6. 
	\end{align}
	By $\norm{A} < 6$, $\alpha > \frac{3}{4}$, and $\rho<\frac{1}{41}$, we can further obtain that $\rho^2 \le \alpha / (4\norm{\Ab})$.
	By the definition of $V^k$ and Eqn.~\eqref{eq:Mat}, we can obtain that 
	\begin{equation}
		\label{eq:ezt}
		\begin{aligned}
			&\EE\left[\norm{\zb^k}\right] 
			\\
			\le& 
			\rho^2/9 
			\cdot
			\sum_{i=1}^k \left(2\rho^2\norm{\Ab}\right)^{k-i}(V^i+V^{i-1})\\
			&+ \left(2\rho^2\norm{\Ab}\right)^k\norm{\zb^0}
			\\
			\overset{\eqref{eq:V_ass}}{\le}&
			\rho^2/9\cdot \sum_{i=1}^k \left(2\rho^2\norm{\Ab}\right)^{k-i} (\alpha^{i} + \alpha^{i-1}) (V^0+\norm{\zb^0})\\
			&+ \left(2\rho^2\norm{\Ab}\right)^k\norm{\zb^0}
			\\
			\le&\frac{\rho^2(\alpha+1)}{9\alpha} \cdot \left(\frac{\alpha}{2}\right)^k\sum_{i=1}^k \left(\frac{1}{2}\right)^{-i}(V^0+\norm{\zb^0})\\
			&+ \left(2\rho^2\norm{\Ab}\right)^k\norm{\zb^0}
			\\
			=& \frac{\rho^2(\alpha+1)}{9\alpha} \cdot \left(\frac{\alpha}{2}\right)^k \cdot (2^{k+1} - 2)(V^0+\norm{\zb^0})\\ 
			&+ \left(2\rho^2\norm{\Ab}\right)^k\norm{\zb^0}
			\\
			\le&
			\left(\frac{4\rho^2 \alpha^{k-1}}{9}+  \left(2\rho^2\norm{\Ab}\right)^k\right) 
			\cdot
			(V^0+\norm{\zb^0}),
		\end{aligned}
	\end{equation}
	where the third inequality is because $\rho^2 \le \alpha / (4\norm{\Ab})$ and the last inequality is because of $\alpha<1$ and $V^0$ is non-negative.
	Furthermore, we can obtain
	\begin{equation}
	\label{eq:vhs}
	\begin{aligned}
     &\mathbb{E}\left[\frac{3.3}{m}\left\|\mathbf{x}^{k}-\mathbf{1} \bar{x}^{k}\right\|^{2} +\frac{2.5\eta^2}{m}\|\mathbf{s}^k-\mathbf{1}\bar{\mathbf{s}}^k\|^2\right]\\
        \le& 6.6	\EE\left[\norm{\zb^k}\right] \\
		\overset{\eqref{eq:ezt}}{\le}&
		\left(2.94\rho^2 \alpha^{k-1}+  6.6\left(2\rho^2\norm{\Ab}\right)^k\right)
			\cdot
			(V^0+\norm{\zb^0})
    \end{aligned}
	\end{equation}
    and
    \begin{equation}
	\label{eq:vhhs}
	\begin{aligned}
    &\mathbb{E}\left[\frac{1}{\sqrt{m}}(\left\|\mathbf{x}^{k}-\mathbf{1} \bar{x}^{k}\right\|+\eta\|\mathbf{s}^k-\mathbf{1}\bar{\mathbf{s}}^k\|)\right]
    \\ 
    \leq& \frac{1}{\sqrt{m}}\left[ \sqrt{\mathbb{E} \|\mathbf{x}^{k}-\mathbf{1} \bar{x}^{k}\|^2}+\sqrt{\mathbb{E}\eta^2\|\mathbf{s}^k-\mathbf{1}\bar{\mathbf{s}}^k\|^2}\right]\\
    \leq& 2\sqrt{\frac{1}{{m}}(\mathbb{E}\|\mathbf{x}^{k}-\mathbf{1} \bar{x}^{k}\|^2+\eta^2\|\mathbf{s}^k-\mathbf{1}\bar{\mathbf{s}}^k\|^2)}\\
    \leq& 2\sqrt{2}\sqrt{\mathbb{E}\|\mathbf{z}^k\|}\\
    		\overset{\eqref{eq:ezt}}{\le}&
     \sqrt{\left(\frac{32 \rho^{2} \alpha^{k-1}}{9}+8\left(2 \rho^{2}\|\mathbf{A}\|\right)^{k}\right) \cdot\left(V^{0}+\norm{\mathbf{z}^{0}}\right)}.
    \end{aligned}
	\end{equation}
    Let $\beta$ denote $\max \left(1-\frac{1}{12 \kappa}, 1-\frac{1}{2 n}\right)$. We are ready to prove the induction inequality:
	\begin{align}
		&\EE\left[V^{k+1}\right] \notag
		\\
		\overset{\eqref{eq:V_dec}}{\le}&
		\max \left(1-\frac{1}{12 \kappa}, 1-\frac{1}{2 n}\right) \cdot V^{k}\notag\\
		&+3\frac{\sqrt{V^{k}}}{ \sqrt{m}}\left(\left\|\mathbf{x}^{k}-\mathbf{1} \bar{x}^{k}\right\|+\left\|\mathbf{s}^{k}-\mathbf{1} \bar{s}^{k}\right\|\right)\notag\\
		&+\frac{3.3}{m}\left\|\mathbf{x}^{k}-\mathbf{1} \bar{x}^{k}\right\|^{2}+\frac{2.5\eta^2}{m}\|\mathbf{s}^k-\mathbf{1}\bar{\mathbf{s}}^k\|^2
		\\
		\overset{\eqref{eq:V_ass},\eqref{eq:vhs},\eqref{eq:vhhs}}{\le}&
        \alpha^k\left(V^0+\|\zb^0\|\right) \notag\\
        &\cdot\left(\right. \beta + 2.94 \rho^2\alpha^{-1} + 6.6\left(2 \rho^{2}\|\mathbf{A}\|\alpha^{-1} \right)^{k} \notag \\
	    &+3\sqrt{\alpha^{-k}(\frac{32 \rho^{2} \alpha^{k-1}}{9}+8\left(2 \rho^{2}\|\mathbf{A}\|\right)^{k})} \left. \right)
		\\
		\le&
		\alpha^k(V^0+\|\zb^0\|) (\beta + 3.92\rho^2 + 6.6(16\rho^2)^k \notag \\
		&+3(\frac{\sqrt{32}\rho\alpha^{-1/2}}{3}+2\sqrt{2}(2\rho^2\|A\| \alpha^{-1})^{\frac{k}{2}}))
		\label{eq:rho_1}
		\\
		\le&
		\alpha^k\left(V^0+\|\zb^0\|\right) (\beta + 0.0239\rho \notag\\
		&+ 0.1072\rho+ \frac{2\sqrt{32}}{\sqrt{3}} \rho + 24\sqrt{2}\rho)
		\label{eq:rho_2}
		\\
		\le&
		\alpha^k\left(V^0+\norm{\zb^0}\right)\left(\max \left(1-\frac{1}{12 \kappa}, 1-\frac{1}{2 n}\right) + 41 \rho\right)
        \notag
		\\
		\le&
		\alpha^{k+1}\left(V^0+\norm{\zb^0}\right).
		\notag
	\end{align}
	Inequality~\eqref{eq:rho_1} is because $\alpha \ge \frac{3}{4}$ and $\sqrt{a^2+b^2}\leq|a|+|b|$. 
	Inequality~\eqref{eq:rho_2} is because the condition~\eqref{eq:rho_cond} implies $\rho < 1/41$ and $\rho^2 \le \alpha / (4\norm{\Ab})$. The last inequality is because of the condition $\rho \le \frac{1}{41}\cdot\min\left(\frac{1}{24\kappa},\frac{1}{4n} \right)$. Thus, Eqn.~\eqref{eq:V_ass} also holds for $t = k+1$ and we complete the proof.
	
	Furthermore, using Eqn.~\eqref{eq:ezt} and $\rho^2 \le \alpha / (4\norm{\Ab})$, we can obtain that
	\begin{equation}
		\begin{aligned}
			&\EE\left[\max\left(\frac{1}{m}\norm{\xb^t - \mathbf{1} \bbx^t}^2, \frac{1}{144mL^2}\norm{\bs^t - \mathbf{1}\bbs^t}^2\right)\right]
			\\
			\le&
			\left(\frac{4 \rho^{2} \alpha^{t-1}}{9}+\left(2 \rho^{2}\|\mathbf{A}\|\right)^{t}\right) \cdot\left(V^{0}+\norm{\mathbf{z}^{0}}\right)
			\\
			\le&
			\left(\frac{1}{45387} \min \left(\frac{1}{36 \kappa^{2}}, \frac{1}{n^{2}}\right)+2^{-t}\right) \cdot\\
			&\max \left(1-\frac{1}{24 \kappa}, 1-\frac{1}{4 n}\right)^{t} \cdot\left( V^{0}+\left\|\mathbf{z}^{0}\right\|\right).
		\end{aligned}
	\end{equation}
\end{proof}
\begin{figure*}[!ht]
\centering
\subfloat{\includegraphics[width=2in]{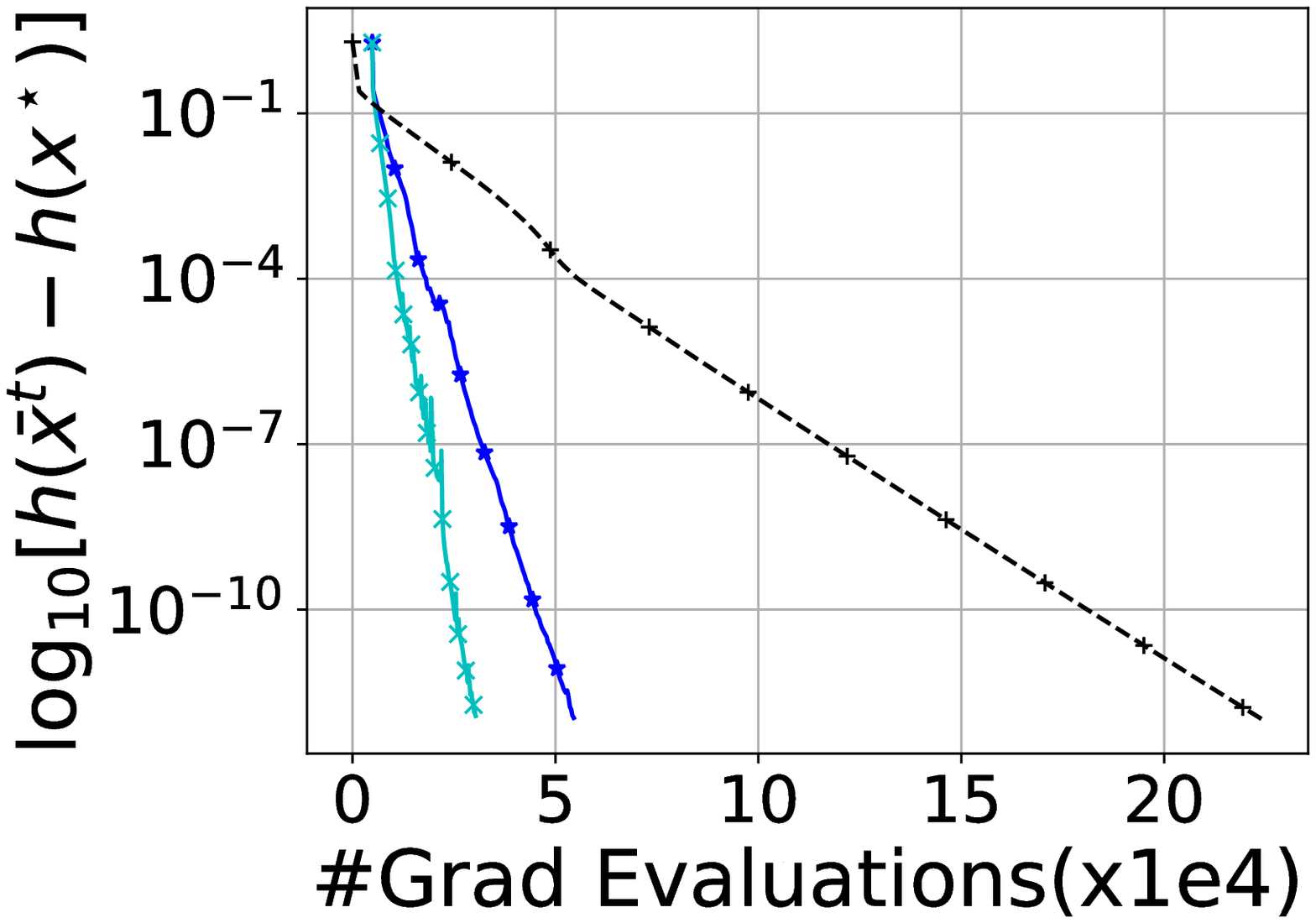}
}
\hfil
\subfloat{\includegraphics[width=2.05in]{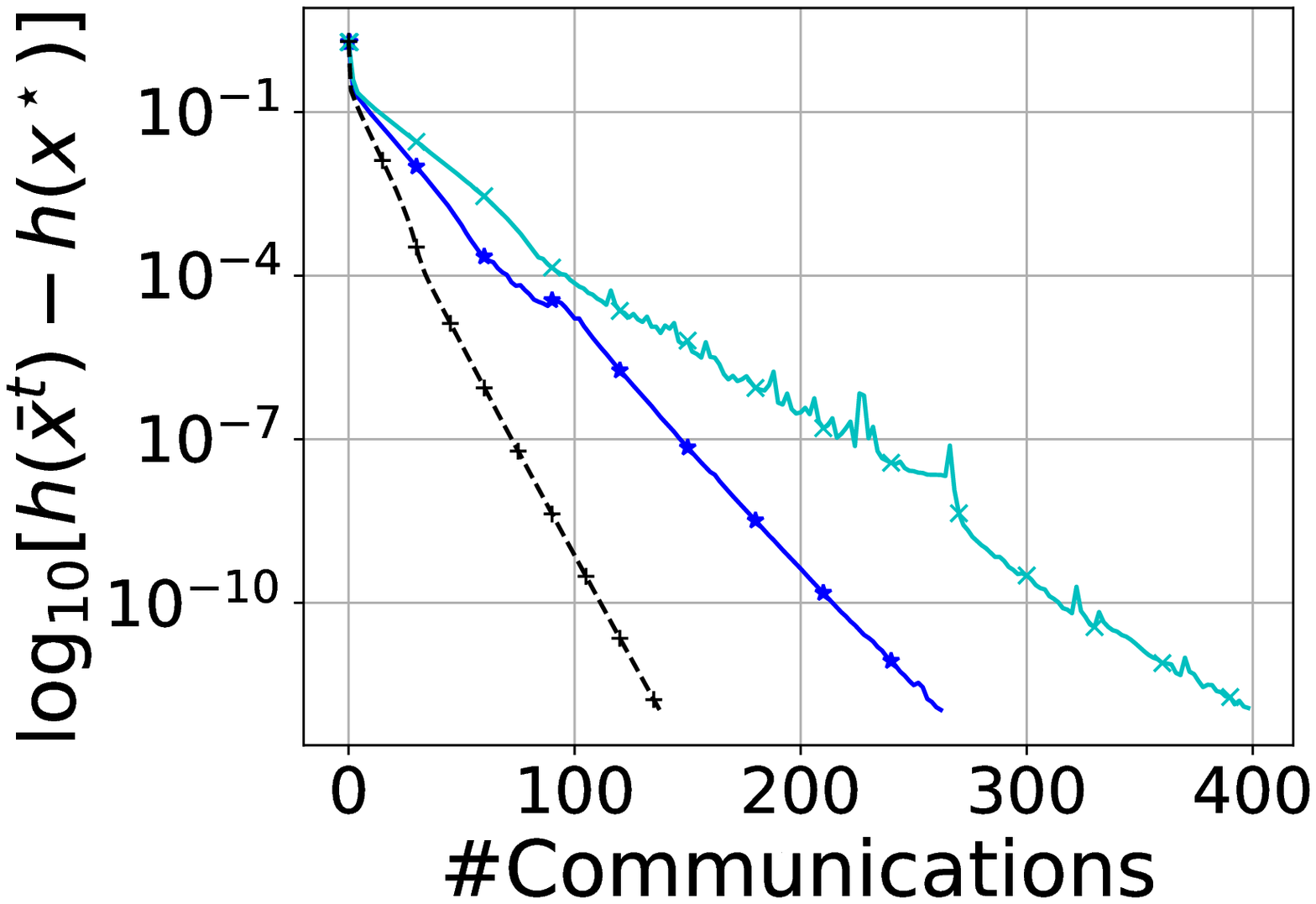}
}
\subfloat{\includegraphics[width=2.01in]{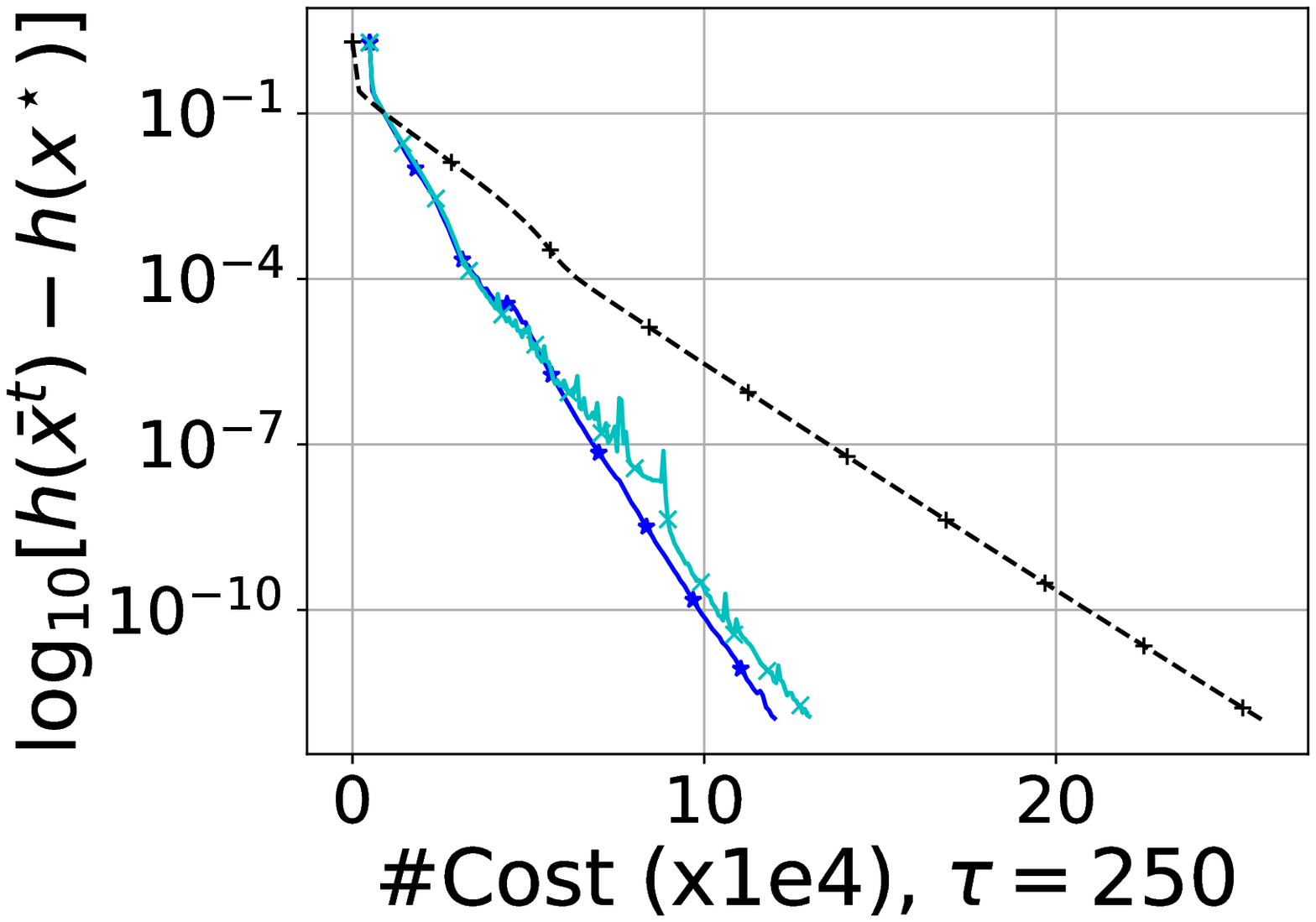}
}
\quad
\subfloat{\includegraphics[width=2in]{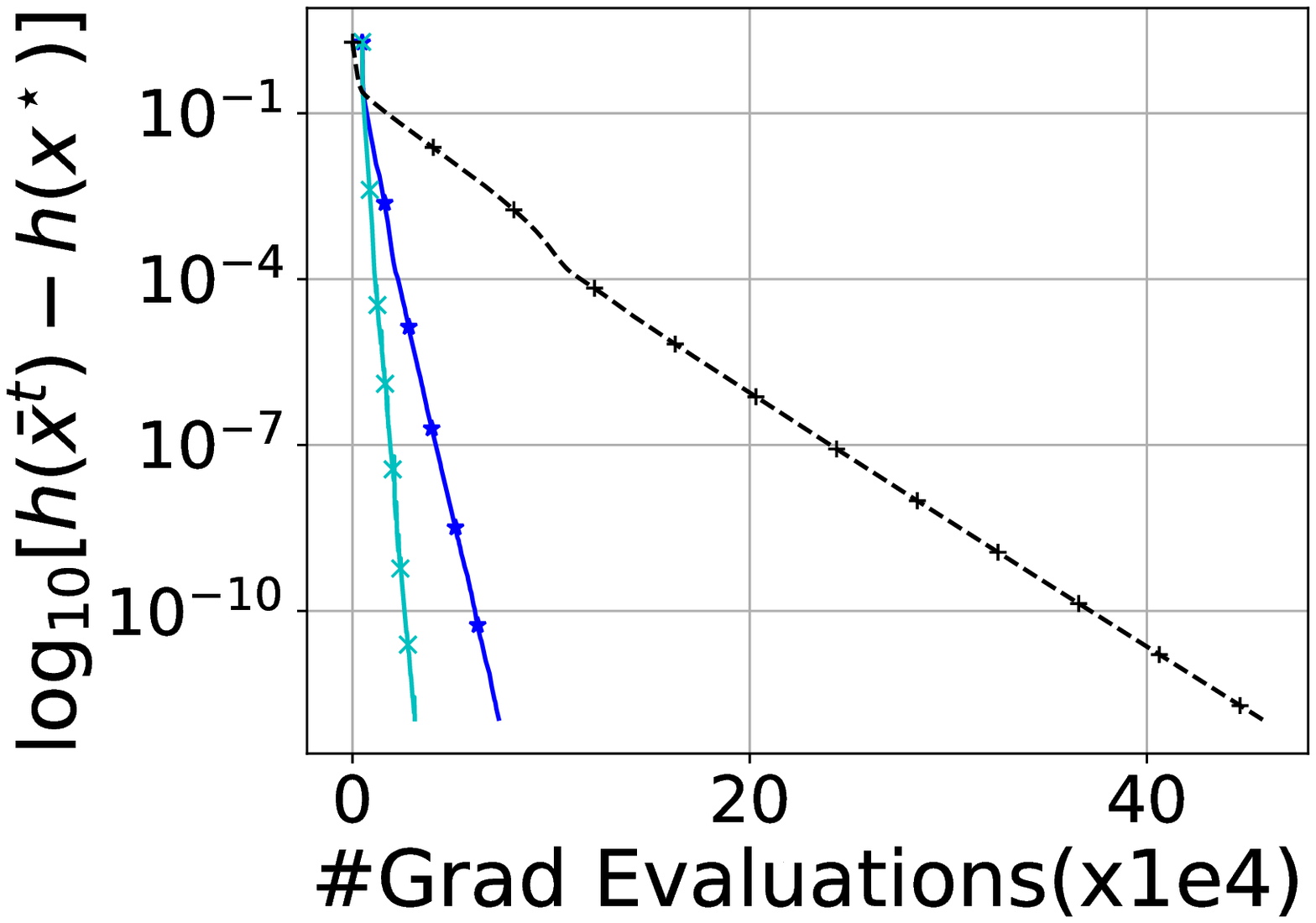}
}
\hfil
\subfloat{\includegraphics[width=2.03in]{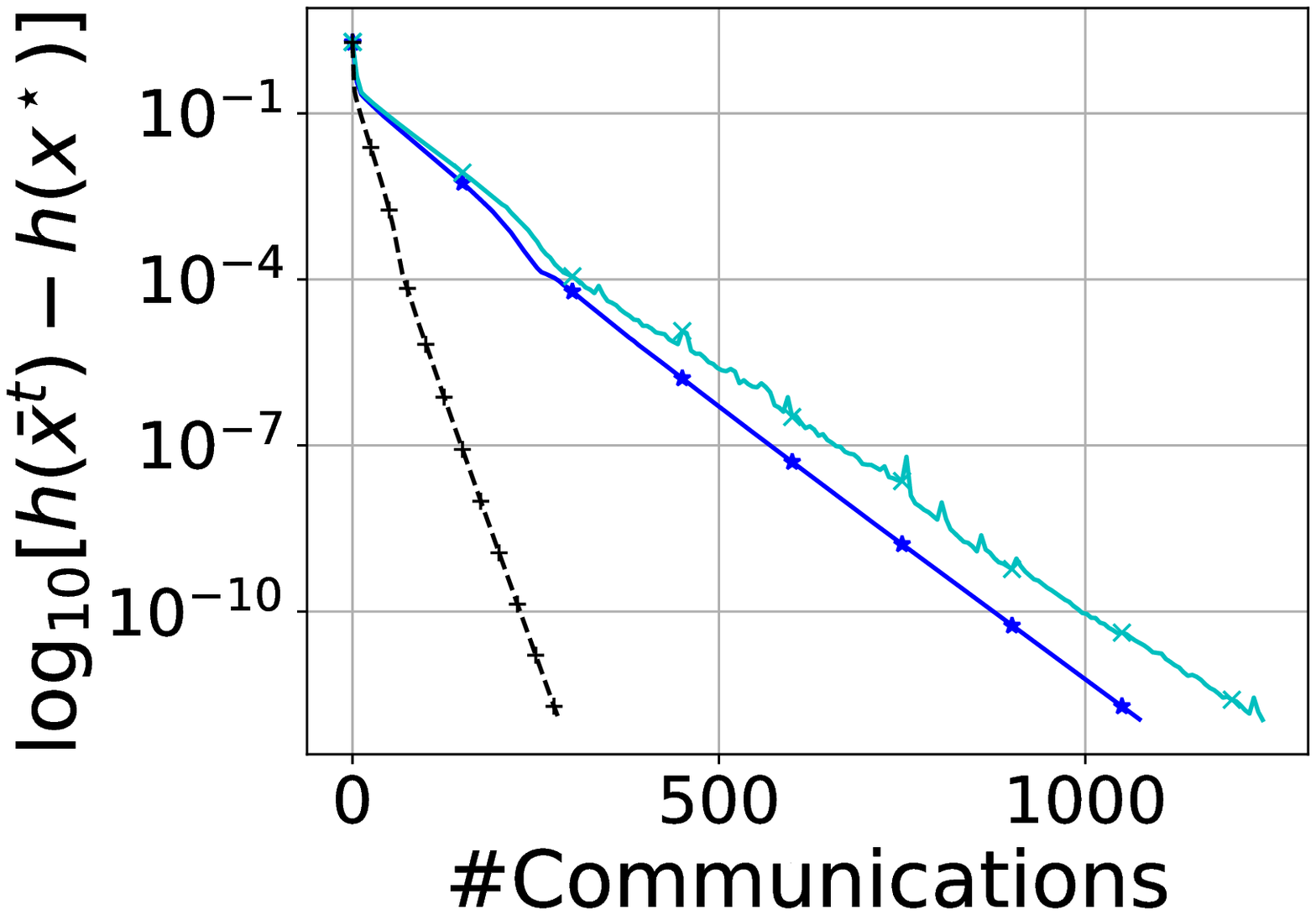}
}
\subfloat{\includegraphics[width=2.01in]{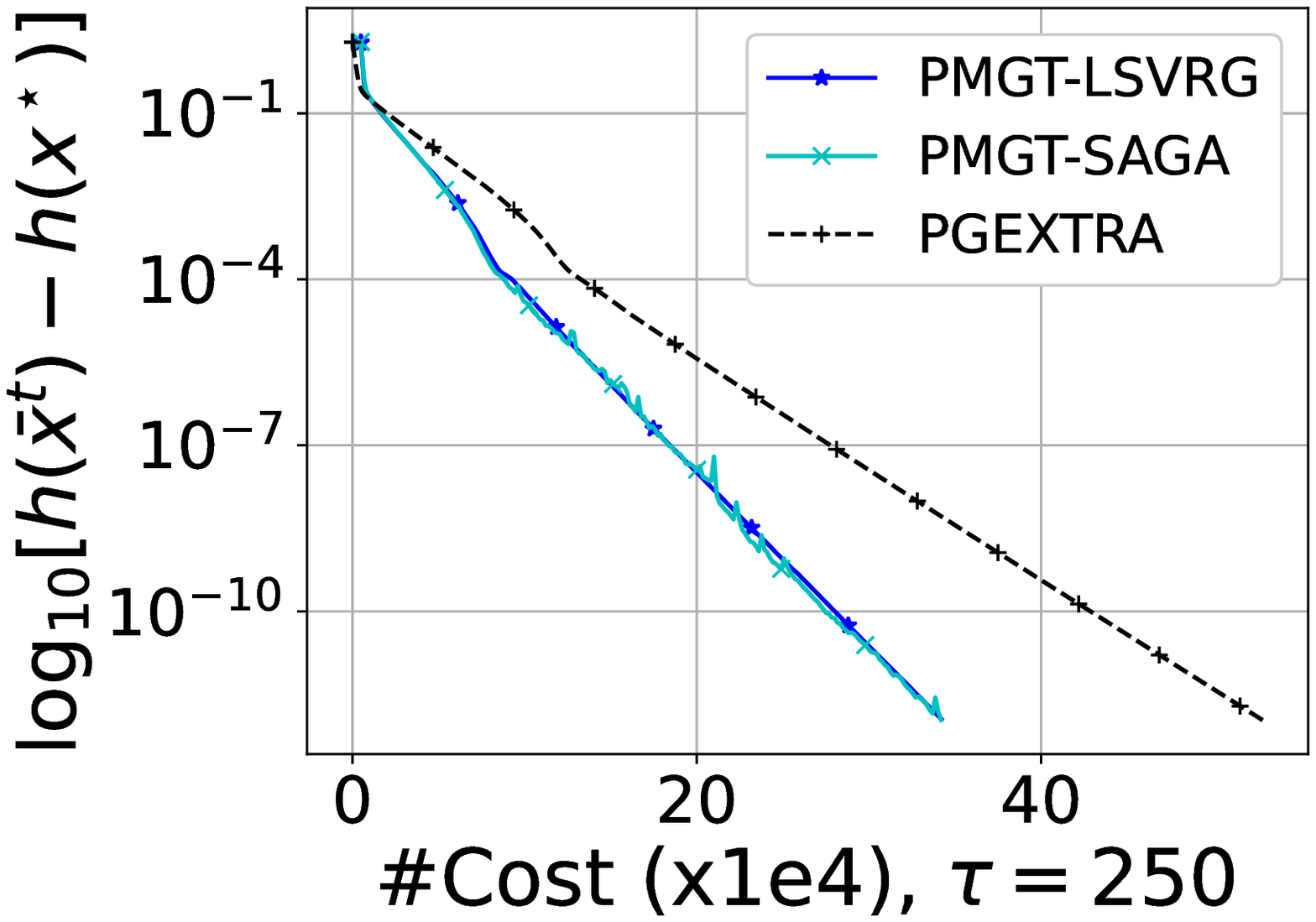}
}
\caption{Performance comparison of \texttt{PMGT-SAGA, PMGT-LSVRG}, and \texttt{PG-EXTRA} with $\sigma=10^{-5}n$. The top row and the bottom row present the results with $1-\lambda_2(W)=0.81$ and $1-\lambda_2(W)=0.05$, respectively.}
\label{fig:e5_005}
\end{figure*}

\begin{figure*}[!ht]
\centering
\subfloat{\includegraphics[width=1.97in]{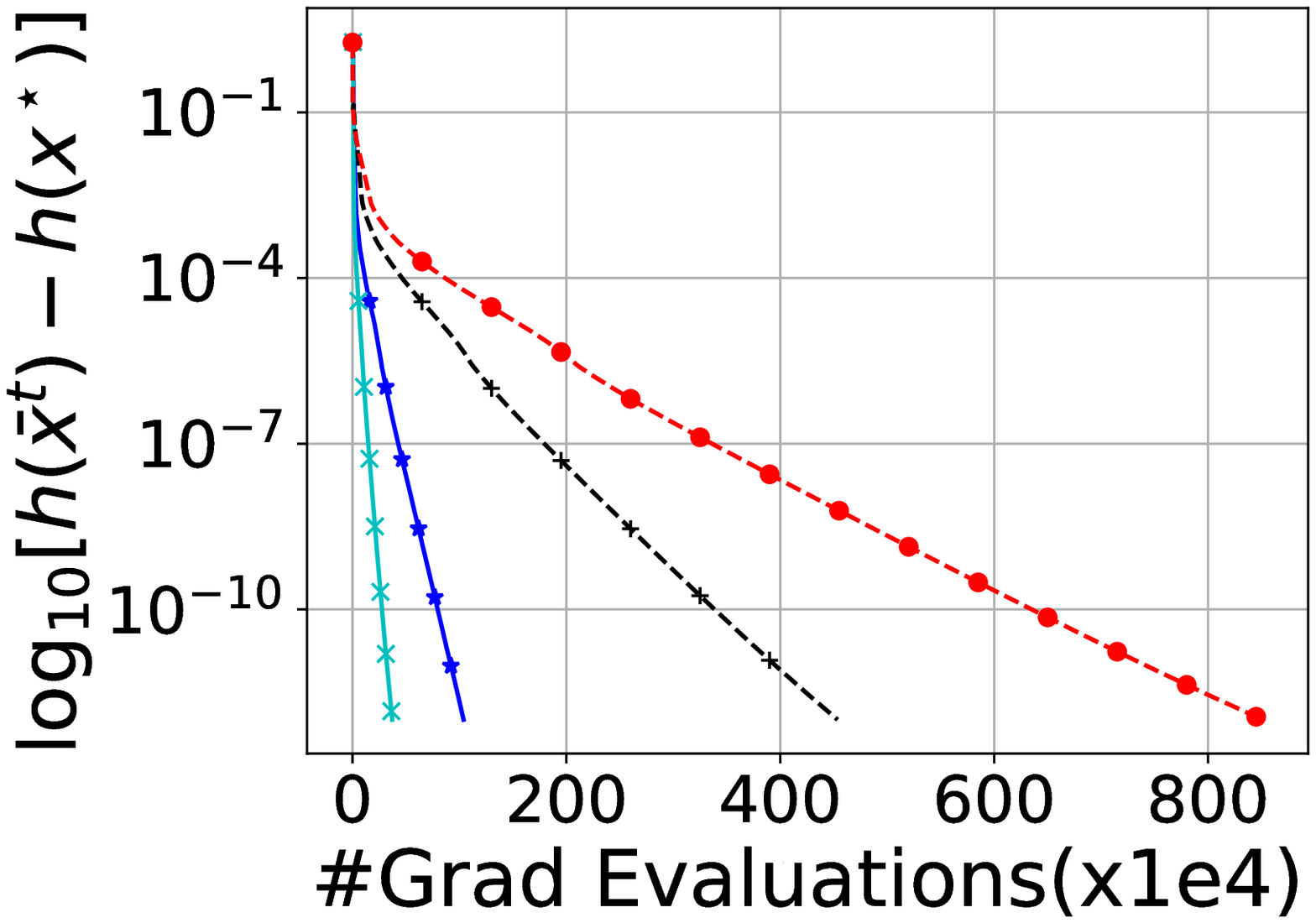}
}
\hfil
\subfloat{\includegraphics[width=2.05in]{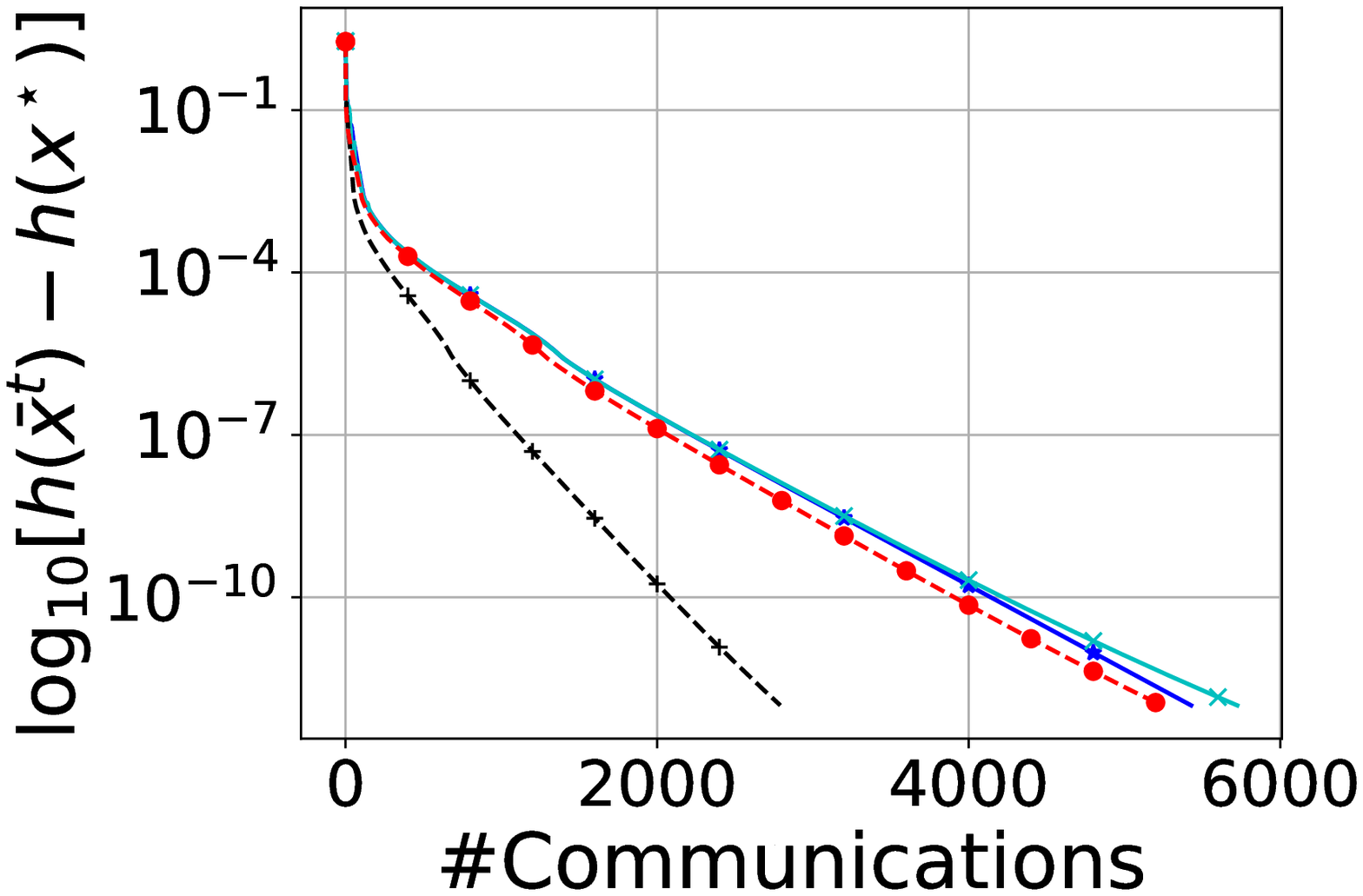}
}
\subfloat{\includegraphics[width=2in]{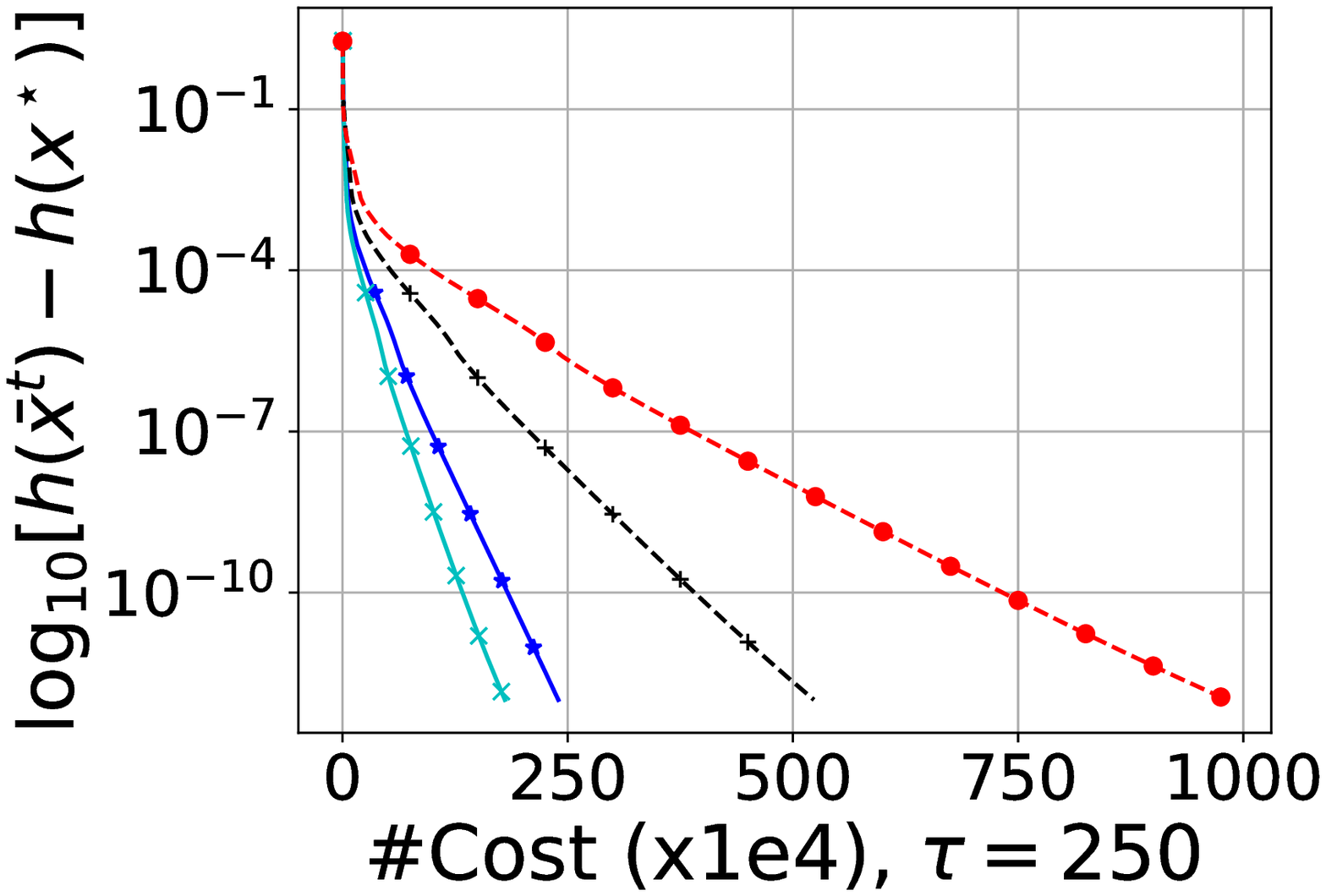}
}
\quad
\subfloat{\includegraphics[width=1.97in]{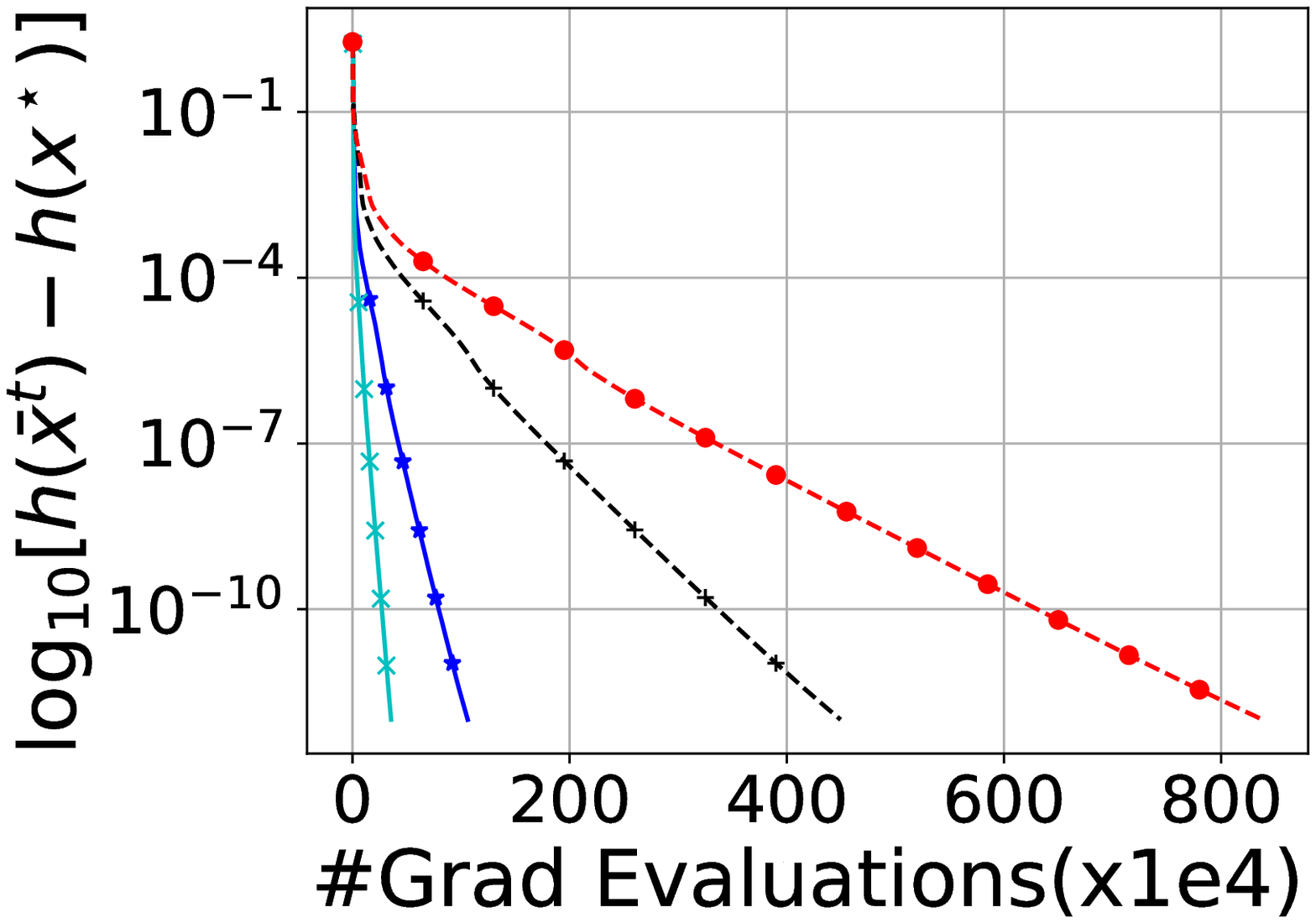}
}
\hfil
\subfloat{\includegraphics[width=2in]{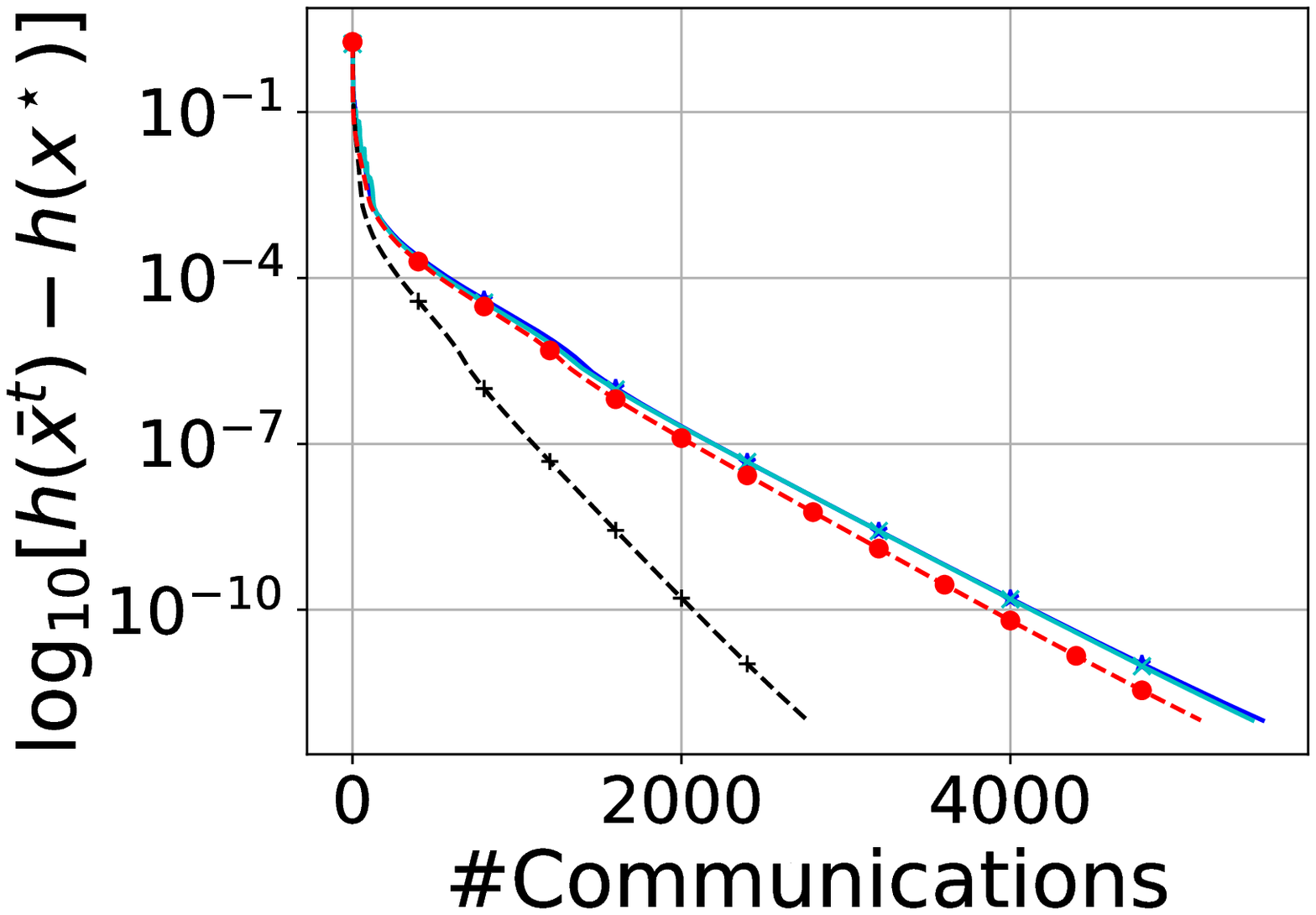}
}
\subfloat{\includegraphics[width=2.08in]{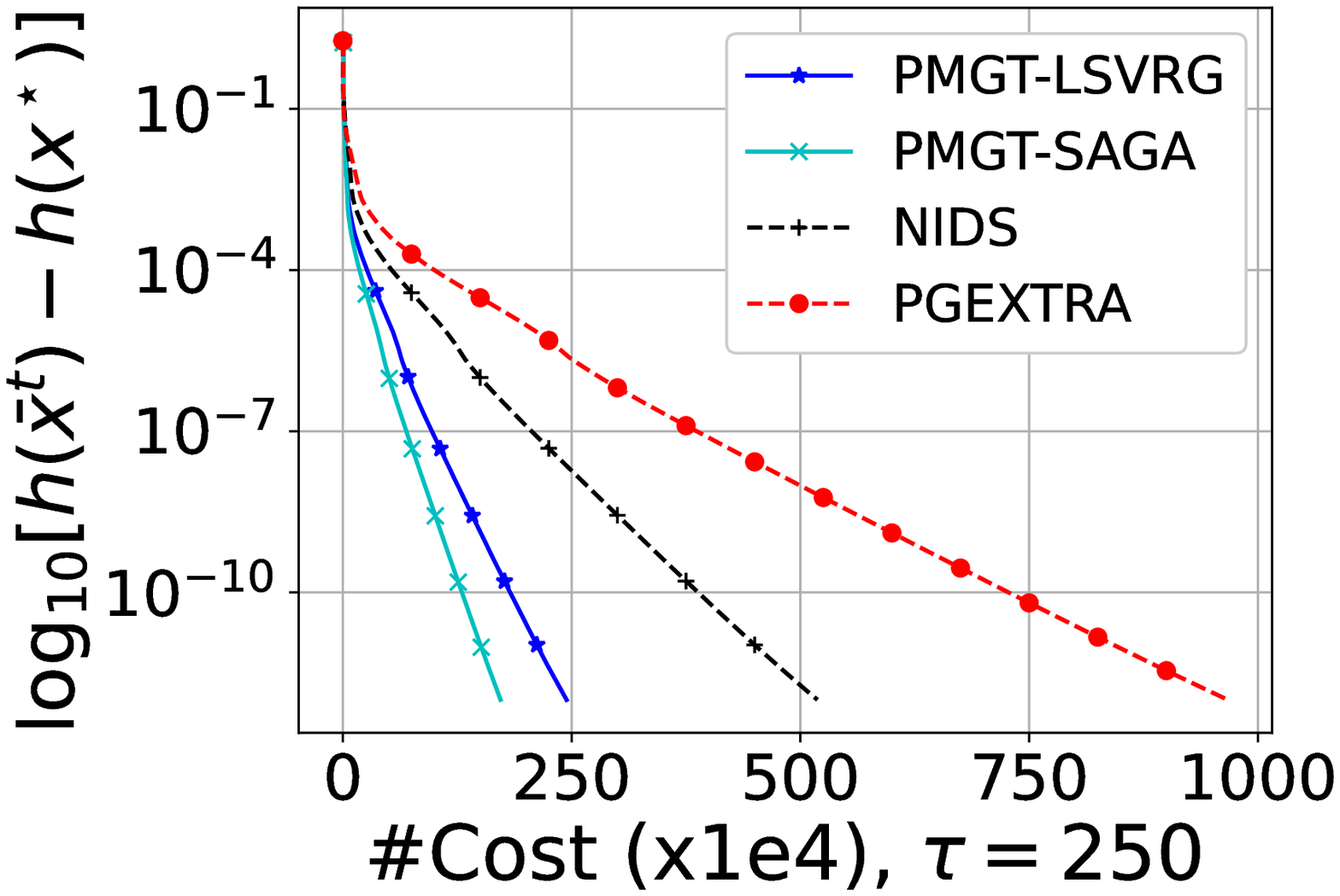}
}
\caption{Performance comparison of \texttt{PMGT-SAGA, PMGT-LSVRG, NIDS}, and \texttt{PG-EXTRA} with $\sigma=10^{-7}n$. The top row and the bottom row present the results with $1-\lambda_2(W)=0.81$ and $1-\lambda_2(W)=0.05$, respectively.}
\label{fig:e7_005}  
\end{figure*}
\begin{figure*}[!ht]
\centering
\subfloat{\includegraphics[width=2in]{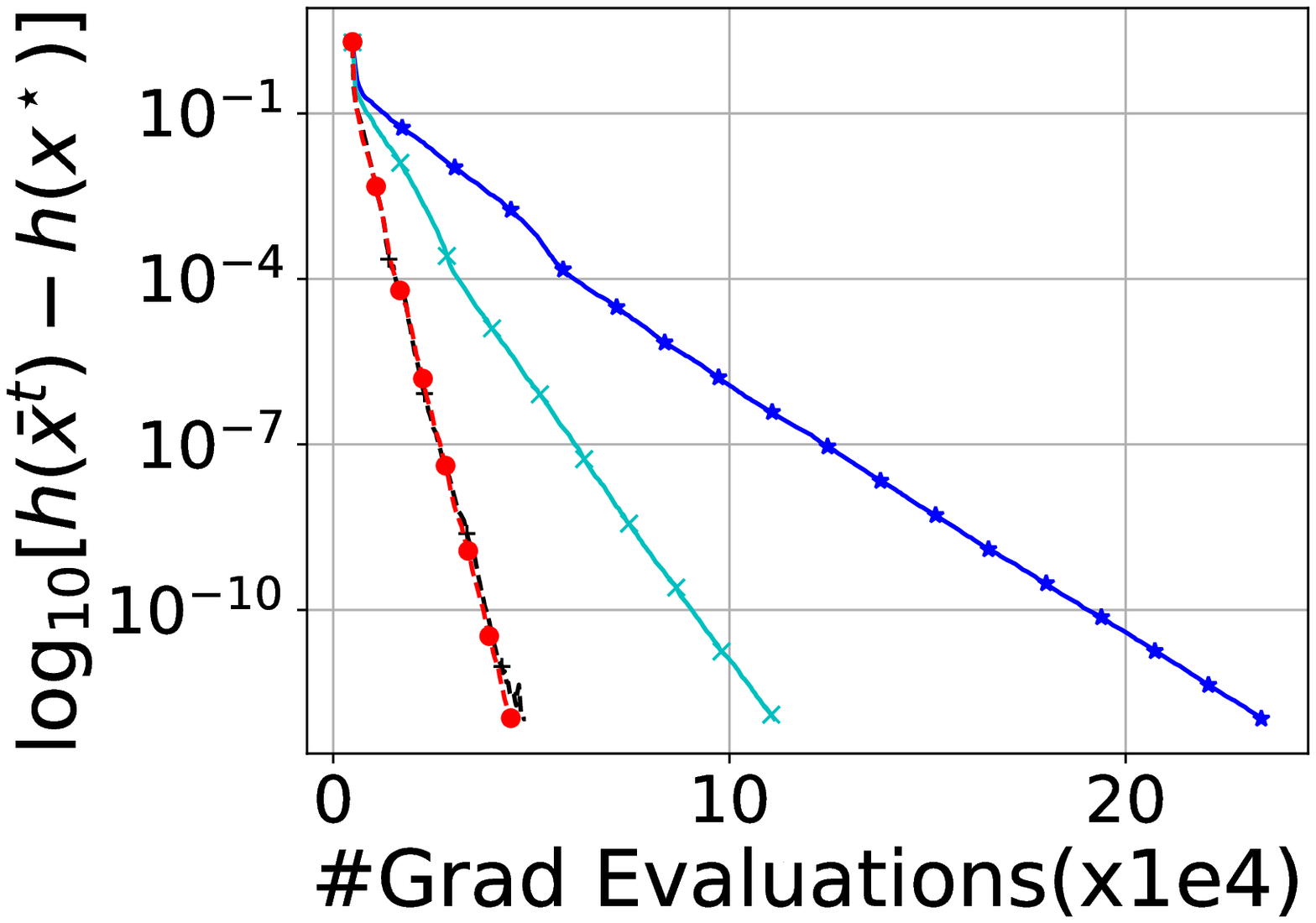}
}
\subfloat{\includegraphics[width=2in]{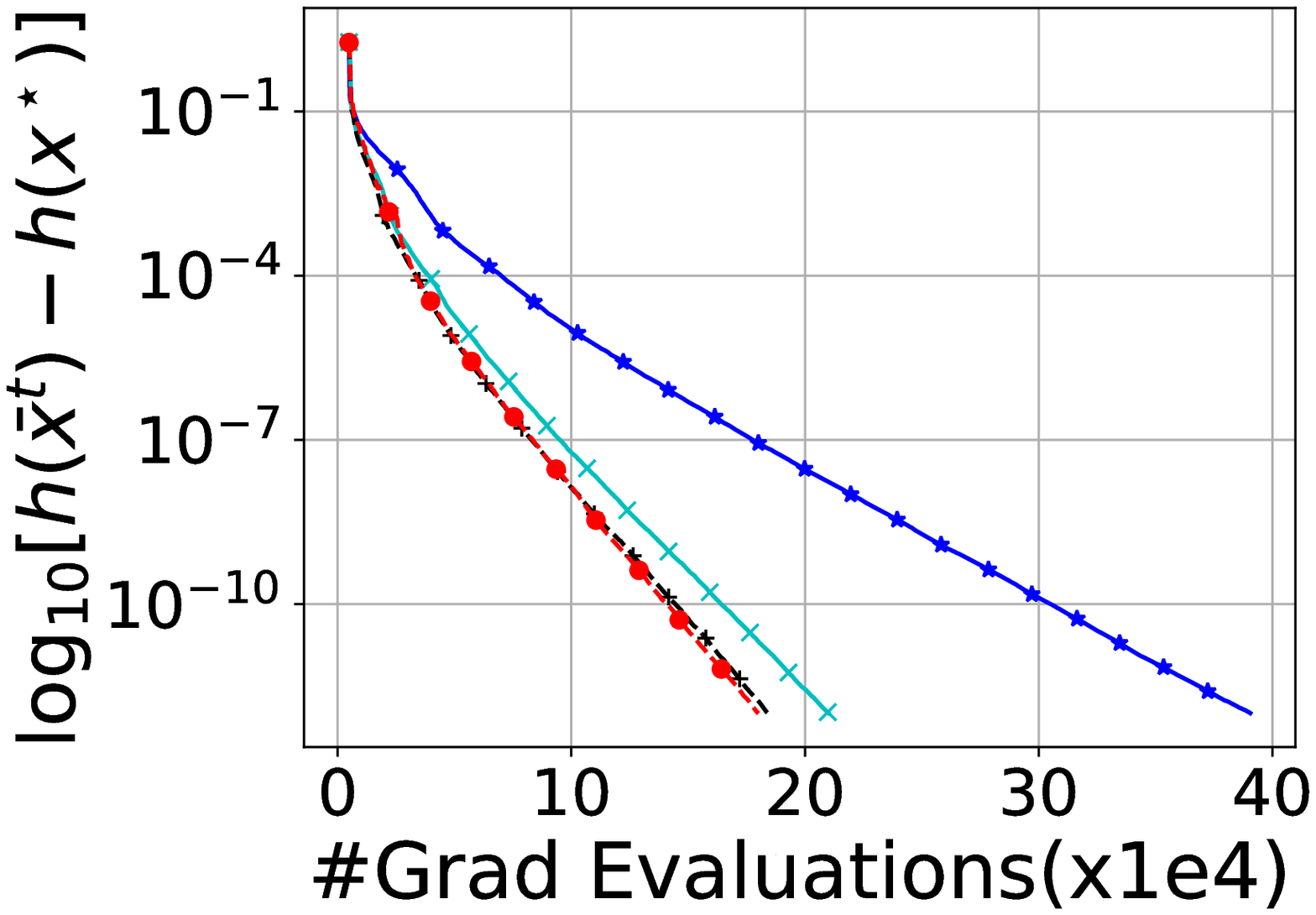}
}
\subfloat{\includegraphics[width=2in]{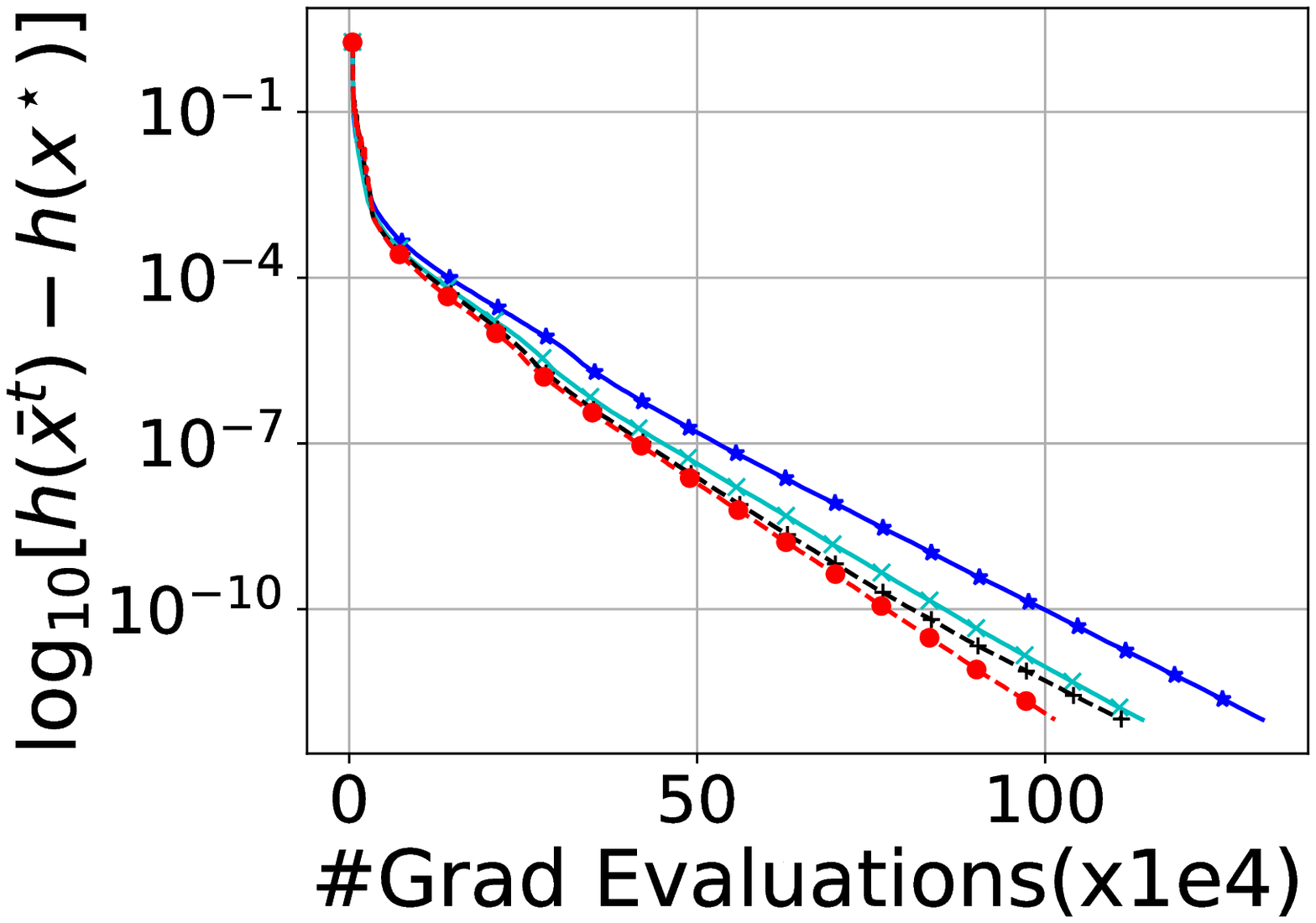}
}
\quad
\subfloat{\includegraphics[width=2in]{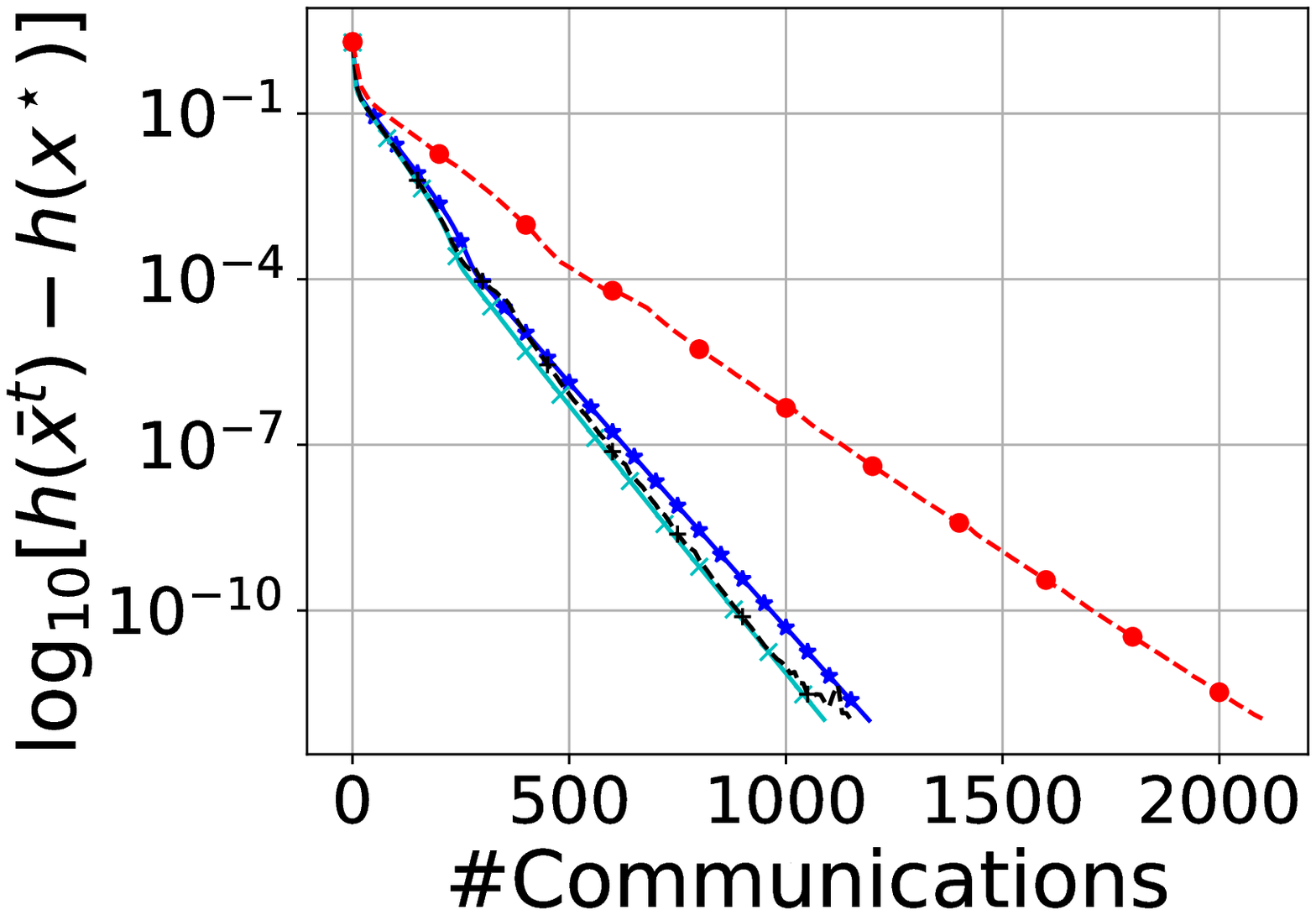}
}
\subfloat{\includegraphics[width=2in]{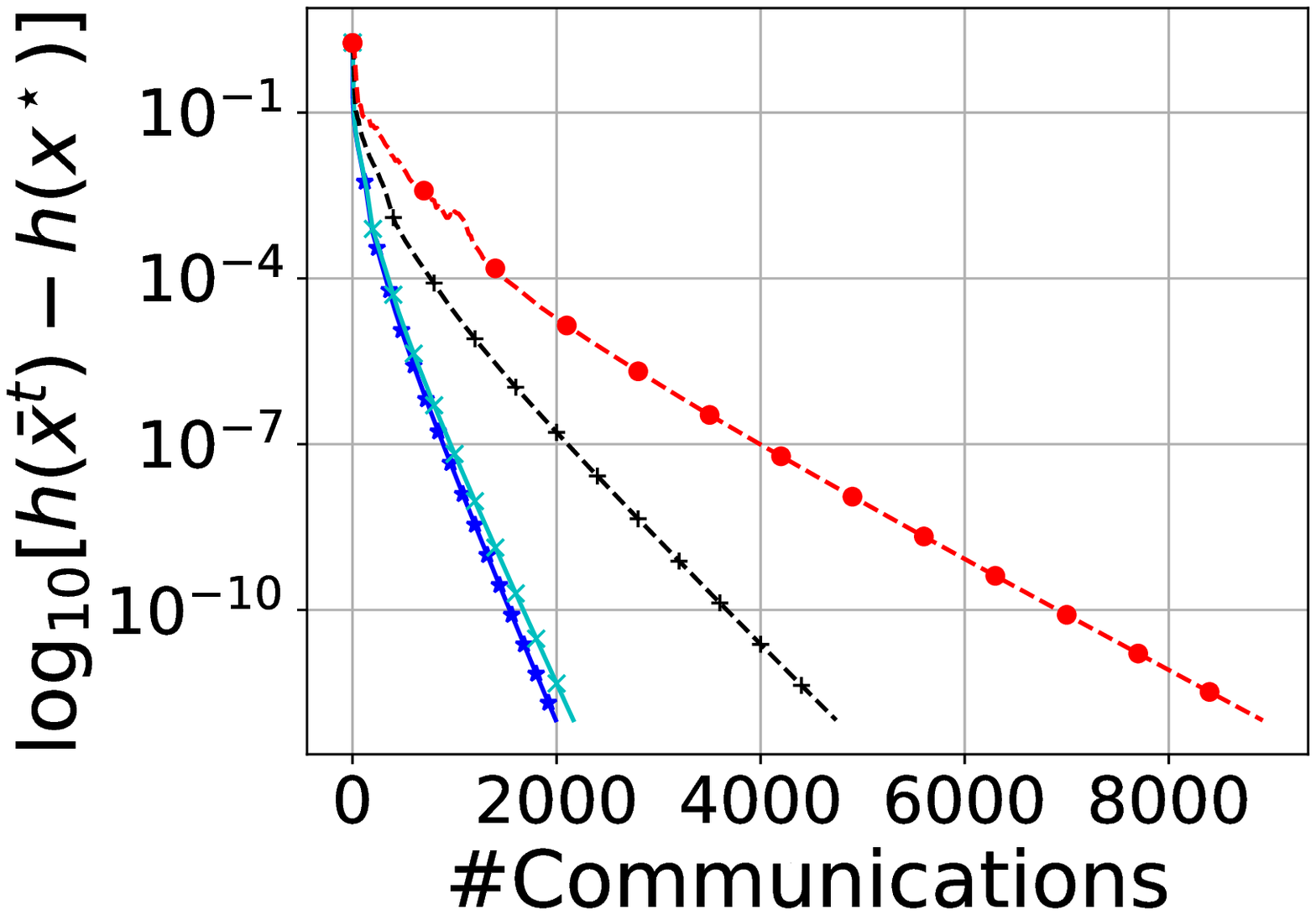}
}
\subfloat{\includegraphics[width=2in]{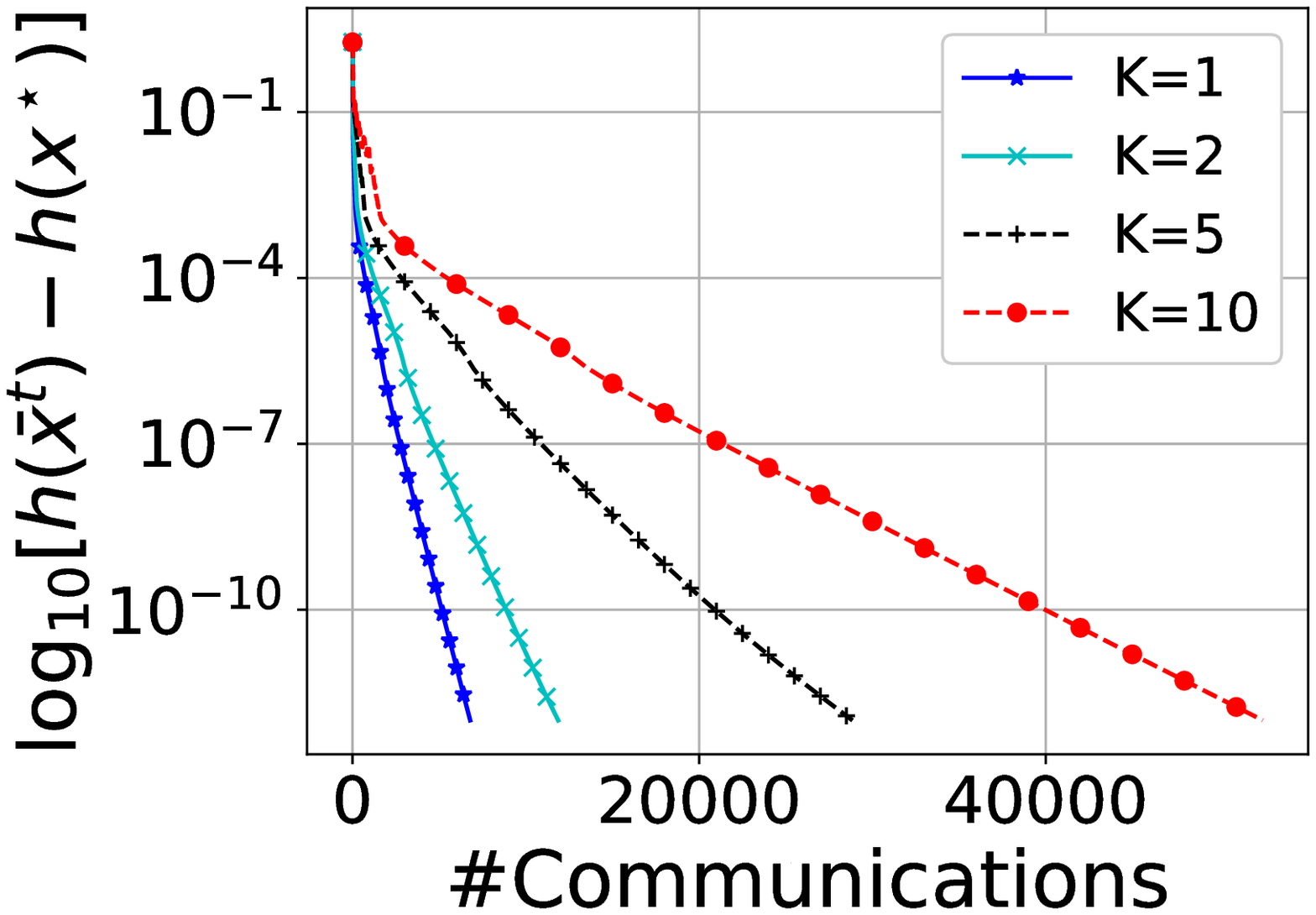}
}
    \caption{Comparisons under different consensus steps $K$ with $1-\lambda_2(W)=0.05$ for \texttt{PMGT-LSVRG}. From left to right, $\sigma=10^{-5}n,10^{-6}n,10^{-7}n$, respectively.} 
    \label{fig:comp_K}
\end{figure*}
\section{NUMERICAL EXPERIMENT} \label{sec:exp}
In this section, we present several numerical experiment results.  We evaluate the performance of the proposed algorithms using logistic regression with $L_1$-regularization for binary classification:
$$f_i(x)=\frac{1}{n}\sum_{j=1}^n \log[1+\exp(-b_j\langle a_j, x\rangle)]+\frac{\sigma_i}{2}\|x\|^2,$$
and 
$$r(x)=\frac{1}{mn}\|x\|_1.$$
We conduct experiments with a real-world dataset a9a ($mn=32560, d=123$) and set $m=20$ which leads to $n=1628$. We consider a random network where each pair of agents is connected with probability $p$ and set $W=I-\frac{L}{\lambda_1(L)}$ where $L$ is the Laplacian of the generated graph. We test our algorithms with different networks where we observe that $1-\lambda_2(W) = 0.05$ and $1-\lambda_2(W)=0.81$, respectively.

We will compare the performance of \texttt{PMGT-SAGA} and \texttt{PMGT-LSVRG} with \texttt{PG-EXTRA} \cite{shi2015pgextra} and \texttt{NIDS} \cite{li2019nids} because the contributions of \cite{Xu2020Unified} mainly focus on the theoretical analysis and the authors show that \texttt{NIDS} achieves a faster convergence rate as compared to the algorithm proposed in \cite{alghunaim2020decentralized}. Moreover, the empirical results reported in \cite{alghunaim2020decentralized} also demonstrate the superiority of \texttt{NIDS}. We will set $\sigma_i \in \{10^{-5}n,10^{-6}n,10^{-7}n\}$ for all agents to control the condition number of $f(x)$. The $y$ axis is the suboptimality $h(\bar{x}^t)-h(x^*)$ and $h(x^*)$ is approximated by the minimal loss over all iterations and all algorithms. The left plot and middle plot show the suboptimality with respect to the number of component gradient evaluations and the number of communications, respectively. The right plot is taken with respect to the cost (evaluating $\nabla f_{i,j}$ is of cost $1$ and the cost of one round of decentralized communication is $\tau$ \cite{hendrikx2020dual}). All parameters are well-tuned and mini-batch methods are used to accelerate convergence.

Figure~\ref{fig:e5_005} reports the results for $\sigma_i=n10^{-5}$ when $1-\lambda_2(W)=0.81$ and $1-\lambda_2(W)=0.05$, respectively. Because \texttt{NIDS} and \texttt{PG-EXTRA} perform similarly, we only report the results of \texttt{PG-EXTRA}. We see that \texttt{PMGT-SAGA} and \texttt{PMGT-LSVRG} are of a higher communication complexity but require much less component gradient evaluations which illustrates the benefits of stochastic algorithms over algorithms based on full gradients. The figures also show that \texttt{PG-EXTRA} converges much slower when $1-\lambda_2(W)$ is small while \texttt{PMGT-SAGA} and \texttt{PMGT-LSVRG} are rather robust due to the extra rounds of decentralized communications. Therefore, multi-consensus offers more advantages when the network is poorly connected.

Figure \ref{fig:e7_005} reports the results for $\sigma_i=n10^{-7}$ when $1-\lambda_2(W)=0.81$ and $1-\lambda_2(W)=0.05$, respectively. In this ill-conditioned case, we see that \texttt{PMGT-SAGA} and \texttt{PMGT-LSVRG} further outperform \texttt{NIDS} and \texttt{PG-EXTRA} with respect to the number of component gradient evaluations. Furthermore, the total rounds of decentralized communications of \texttt{PMGT-VR} algorithms are roughly the same as that of \texttt{PG-EXTRA}. \texttt{NIDS} is faster than \texttt{PG-EXTRA} since we observe that \texttt{NIDS} can converge with a larger stepsize. In this setting, the condition number $\kappa$ is the dominant factor and the impact of network is relatively small.

We also see that the costs of \texttt{PMGT-SAGA} and \texttt{PMGT-LSVRG} are much lower than those of \texttt{PG-EXTRA} and \texttt{NIDS} regardless of the experiment setting if we choose $\tau=250$ and they perform similarly when $\tau$ is large enough. For instance, \texttt{PMGT-SAGA} and \texttt{PG-EXTRA} perform similarly for $\tau\approx 1300$ in the first row of Fig~\ref{fig:e5_005} and for $\tau\approx 500$ in the second row; \texttt{PMGT-SAGA} and \texttt{NIDS} perform similarly for $\tau \approx 1400$ in the two experiments with $\sigma=10^{-7}n$. We observe that decentralized optimization algorithms are faced with a computation-communication tradeoff.
$\tau$ can be considered as the relative ratio of cost of communication and cost of gradient evaluation. When $\tau$ is not too big (decentralized communication is not too expensive), \texttt{PMGT-VR} methods are preferable as compared to \texttt{PG-EXTRA} and \texttt{NIDS}. 

Finally, we compare our algorithms with different number of per-iteration decentralized communications $K$, reported in Figure \ref{fig:comp_K}. Because the behaviors of \texttt{PMGT-SAGA} and \texttt{PMGT-LSVRG} are similar with different $K$, we only report the results of \texttt{PMGT-LSVRG}. We consider a network with $1-\lambda_2(W) = 0.05$ and set $\sigma_i=\sigma=10^{-5}n,10^{-6}n,10^{-7}n$ for three experiments respectively. We see that when $K$ starts to increase, \texttt{PMGT-LSVRG} converges faster with respect to the number of component gradient evaluations and the communication complexity is relatively stable in the first two columns because the faster convergence rate compensates for the extra communication cost. We also see that after a certain threshold, a larger $K$ cannot further improve the convergence rate with respect to the number of component gradient evaluations and the total rounds of decentralized communications increase significantly. For the ill-conditioned case, we see again that the condition number is the dominant factor and the impact of multi-consensus is relatively small. 

\section{Conclusion} \label{sec:conclude}
In this paper, we propose a novel algorithmic framework, called \texttt{PMGT-VR}, which is based on a novel combination of multi-consensus, gradient tracking and variance reduction techniques.
Theoretically, we show that a decentralized proximal variance reduction algorithm could admit a similar convergence rate by imitating its centralized counterpart.
Specifically, we propose two representative algorithms named \texttt{PMGT-SAGA, PMGT-LSVRG}. 
We also show that \texttt{PMGT-SAGA} and \texttt{PMGT-LSVRG} can achieve much better computational efficiency than existing decentralized proximal algorithms.
To the best of our knowledge, \texttt{PMGT-SAGA} and \texttt{PMGT-LSVRG} are the first decentralized proximal variance reduction algorithms.
Finally, our methodology can be extended to other variance reduction techniques in a similar fashion. 
Thus, our \texttt{PMGT-VR} framework can provide an insight in developing novel decentralized proximal variance reduction algorithms.

\bibliography{ref.bib}
\bibliographystyle{IEEEtran}

\appendices

\section{Proofs of Lemma~\ref{lem:gradient_tracking_mgt}, \ref{lem:fastmix}, \ref{lem:prox}, and \ref{lem:consensus_error_mgt}}

\begin{algorithm}[H]
	\caption{FastMix}
	\label{alg:mix}
	\begin{small}
		\begin{algorithmic}[1]
			\STATE {\bf Input:} $\xb^{0} = \xb^{-1}$, $K$, $W$, stepsize $\eta_w = \frac{1 - \sqrt{1-\lambda^2_2(W)}}{1+\sqrt{1-\lambda^2_2(W)}}$.
			\FOR {$k=0,\dots, K$ }
			\STATE $\xb^{k+1} = (1+\eta_w)W\xb^k - \eta_w \xb^{k-1}$;  
			\ENDFOR
			\STATE {\bf Output:} $\xb^K$.
		\end{algorithmic}
	\end{small}
\end{algorithm}

In this section, we first provide the proof of lemmas~\ref{lem:gradient_tracking_mgt} and~\ref{lem:prox} that are related to gradient tracking and proximal mapping.
The proof of lemma~\ref{lem:fastmix} can be found in \cite{Liu2011FastMix}. Then, we invoke these results to give a proof of Lemma~\ref{lem:consensus_error_mgt} that is related to the iterative inequality about consensus errors.

\begin{proof}[Proof of Lemma~\ref{lem:gradient_tracking_mgt}] By the update rule and property of \texttt{FastMix}:
\begin{equation*}
    \begin{aligned}
\bs^{t} &= \mathrm{FastMix}\left(\bs^{t-1} + \vb^{t} - \vb^{t-1}, K\right),\\   \mathbf{1}^\top\xb &=\mathbf{1}^\top\mathrm{FastMix}(\xb,K),~\forall \xb \in \RR^{m\times d},
\end{aligned}
\end{equation*}
we know that 
 \begin{equation*}
     \begin{aligned}
     \bbs^t = \bbs^{t-1}+\bar{v}^t-\bar{v}^{t-1}.
     \end{aligned}
 \end{equation*}
Since $\bbs^t = \bar{v}^t$ holds for $t=0$, by induction we can show that $\bbs^t = \bar{v}^t$ holds. $\EE[\bbs^t] = \frac{1}{m}\sum_i^m \nabla f_i(\xb_i^t)$ holds because $\vb_i^t$ is an unbiased estimator of $\nabla f_i(\xb_i^t)$ with \texttt{PMGT-SAGA} and \texttt{PMGT-LSVRG}. We then have
    \begin{equation*}
	\begin{aligned}
	&\norm{\nabla f(\bbx^t) - \EE[\bbs^t]}^2\\
	=&
	\norm{\frac{1}{m}\sum_{i=1}^m \left(\nabla f_i(\xb_i^t) - \nabla f_i(\bbx^t)\right)}^2\\
	\le&
	\frac{1}{m}\sum_{i=1}^m\norm{\nabla f_i(\xb_i^t) - \nabla f_i(\bbx^t)}^2
	\\
	\le&
	\frac{L^2}{m}\sum_{i=1}^m\norm{\xb_i^t - \bbx^t}^2\\
	=&
	\frac{L^2}{m}\norm{\xb^t - \mathbf{1}\bbx^t}^2.
	\end{aligned}
    \end{equation*}
\end{proof}

\begin{proof}[Proof of Lemma~\ref{lem:prox}]
	We recall that $\zb_i,\xb_i$ are the $i$-th row of the matrices. By the definition of the proximal operators, we have
	\begin{align*}
	&\proximal_{\eta m,R}(\xb)\\
	=&\argmin_{\zb\in \RR^{m\times d}} \left( R(\zb)+ \frac{1}{2\eta m}\norm{\zb-\xb}^2\right) \\
	=&
	\argmin_{\zb\in \RR^{m\times d}} \left( \frac{1}{m}\sum_{i=1}^{m}r(\zb_{i})+ \sum_{i=1}^{m}\frac{1}{2\eta m}\norm{\zb_{i}-\xb_i}^2\right) 
	\\
	=&
	\argmin_{\zb\in \RR^{m\times d}} \left( \sum_{i=1}^{m}r(\zb_{i})+ \sum_{i=1}^{m}\frac{1}{2\eta}\norm{\zb_{i}-\xb_i}^2\right) \\
	=&  \begin{pmatrix}
	\argmin_{z \in R^d}\left( r(z)+\frac{1}{2\eta}\norm{z-\xb_{1}}  \right)^\top \\ 
	\vdots \\ 
	\argmin_{z\in R^d}\left( r(z)+\frac{1}{2\eta}\norm{z-\xb_{m}}  \right)^\top
	\end{pmatrix}.
	\end{align*}
	Therefore, we have the following equation
	\begin{equation*}
	\proximal_{\eta m,R}^{(i)}(\xb)=\proximal_{\eta,r}(\xb_{i}).
	\end{equation*} 

    Using the above result, we expand the sum
	\begin{align*}
	&\norm{\proximal_{\eta  m,R}(\frac{1}{m}\mathbf{1}\mathbf{1}^\top \xb)-\frac{1}{m}\mathbf{1}\mathbf{1}^\top\proximal_{\eta  m,R}( \xb)}^2
	\notag
	\\
	\overset{\eqref{eq:local proximal equality}}{=}&
	{m\norm{\proximal_{\eta  ,r}(\frac{1}{m}\mathbf{1}^\top \xb)-\frac{1}{m}\sum_{i=1}^m\proximal_{\eta ,r}( \xb_i)}^2}
	\notag
	\\
	=&
	{m\norm{\frac{1}{m}\sum_{i=1}^m\left(\proximal_{\eta  ,r}( \frac{1}{m}\mathbf{1}^\top \xb)-\proximal_{\eta ,r}( \xb_i)\right) }^2}
	\notag
	\\
	\le&
	m\cdot \frac{1}{m} \sum_{i=1}^m \norm{\proximal_{\eta  ,r}( \frac{1}{m}\mathbf{1}^\top \xb)-\proximal_{\eta ,r}( \xb_i)}^2
	\\
	\le&
	\sum_i^m \norm{\frac{1}{m}\mathbf{1}^\top \xb - \xb_i}^2
	\\
	=&
	\norm{\xb - \frac{1}{m}\mathbf{11^\top} \xb}^2,
	\end{align*}
	where the last inequality is because of the non-expansiveness of proximal operator.
\end{proof}

We are ready to prove the iterative inequality about consensus errors.
\begin{proof}[Proof of Lemma \ref{lem:consensus_error_mgt}]
    For simplicity, we denote \texttt{FastMix}$(\cdot,K)$ operation as $\mathbb{T}(\cdot)$. From Lemma~\ref{lem:fastmix} we can know that
	\begin{equation}\label{eq:mix}
	\norm{\mathbb{T}(\xb)-\frac{1}{m}\mathbf{1}\mathbf{1}^\top\xb}\le \rho\norm{\xb-\frac{1}{m}\mathbf{1}\mathbf{1}^\top\xb}.
	\end{equation}
	First, we have
	\begin{equation*}
	\begin{aligned}
	&\norm{\xb^{t+1} - \mathbf{1}\bbx^{t+1}}\\
	\le & \rho\norm{\proximal_{\eta m,R}(\xb^{t}-\eta \bs^t)-\frac{\mathbf{1}\mathbf{1}^\top}{m}\proximal_{\eta m,R}(\xb^{t}-\eta \bs^t)}\\
	\le &
	    \rho\norm{\proximal_{\eta m,R}(\xb^{t}-\eta \bs^t)- \proximal_{\eta m,R}\left( \mathbf{1}(\bbx^{t}-\eta\bbs^{t})\right)}
	\\&+
    \rho\norm{\proximal_{\eta m,R}\left( \mathbf{1}(\bbx^{t}-\eta\bbs^{t})\right)-\frac{\mathbf{1}\mathbf{1}^\top}{m}\proximal_{\eta m,R}(\xb^{t}-\eta \bs^t)}
    \\
	\le &
	\rho\norm{\xb^{t}-\mathbf{1}\bbx^{t}}+\rho\eta\norm{\bs^t-\mathbf{1}\bbs^{t}}
	\\ &+
	\rho\norm{ \left(\xb^t-\eta\bs^t \right)   -\mathbf{1}\left( \bbx^t-\eta\bbs^t \right) } 
	\\
	\le&
	2\rho\norm{\xb^{t}-\mathbf{1}\bbx^{t}}+2\rho\eta\norm{\bs^t-\mathbf{1}\bbs^{t}},
	\end{aligned}
	\end{equation*}
	where the first inequality is due to the update rule and Eqn.~\eqref{eq:mix}, and the third inequality is because of Lemma~\ref{lem:prox} and the non-expansiveness of proximal operator.
	
	Therefore, we can obtain that
	\begin{equation*}
	\begin{aligned}
	&\frac{1}{m}\norm{\xb^{t+1} - \mathbf{1}\bbx^{t+1}}^2\\
	\le&
	8\rho^2 \frac{1}{m}\norm{\xb^{t}-\mathbf{1}\bbx^{t}}^2
	+
	8\rho^2 \frac{\eta^2}{m} \norm{\bs^t-\mathbf{1}\bbs^{t}}^2.
	\end{aligned}
	\end{equation*}
	
	Furthermore, we have
	\begin{equation*}
	\begin{aligned}
	&\norm{\bs^{t+1} - \mathbf{1}\bbs^{t+1}}^2\\
	=&
	\norm{\mathbb{T}\left(\bs^t + \vb^{t+1} - \vb^t\right) - \frac{1}{m}\mathbf{1}\mathbf{1}^\top\left(\bs^t + \vb^{t+1} - \vb^t\right) }^2
	\\
	\le&
	2\rho^2 \norm{\bs^t - \mathbf{1} \bbs^t}^2
	+
	2\rho^2 \norm{\vb^{t+1} - \vb^t - \frac{1}{m}\mathbf{1}\mathbf{1}^\top(\vb^{t+1} - \vb^t)}^2
	\\
	\le&
	2\rho^2 \norm{\bs^t - \mathbf{1} \bbs^t}^2
	+
	2\rho^2 \norm{\vb^{t+1} - \vb^t}^2,
	\end{aligned}
	\end{equation*}
	where the last inequality is because it holds that $\norm{\xb-\frac{1}{m}\mathbf{1}\mathbf{1}^\top\xb}\le\norm{\xb}$ for any $\xb\in\RR^{m\times d}$.
\end{proof}

\section{Proof of Lemma~\ref{lem:local_gradient_estimator} and \ref{lem:consensus_error}}
In this section, we derive an upper bound for $\norm{\vb^{t+1}-\vb^t}^2$ and substitute it into Lemma~\ref{lem:consensus_error_mgt} to obtain the desired linear system inequality. Before continuing, we first prove two auxiliary results.
\begin{lemma}
    For \texttt{PMGT-VR} methods, it holds that
	\begin{equation}
	\label{eq:bnd_1}
	\begin{aligned}
	&\frac{1}{mn}\sum_{i=1,j=1}^{m,n} \norm{\nabla f_{i,j}(\xb_i^t) - \nabla f_{i,j}(x^*)}^2\\
	\le&
	4L \cdot D_f(\bbx^t, x^*)
	+
	\frac{2L^2}{m}\norm{\xb^t - \mathbf{1}\bbx^t}^2,
	\end{aligned}
	\end{equation}
		and
	\begin{equation}
	\label{eq:bnd_2}
	\begin{aligned}
	&\frac{1}{mn}\sum_{i=1,j=1}^{m,n} \norm{\nabla f_{i,j}(\xb_i^t) - \nabla f_{i,j}(x^*)}^2\\
	\le&
	2L^2\norm{\bbx^t - x^*}^2
	+
	\frac{2L^2}{m}\norm{\xb^t - \mathbf{1}\bbx^t}^2.
	\end{aligned}
	\end{equation}
\end{lemma}
\begin{proof}
We have
    \begin{equation*}
	\begin{aligned}
	&\frac{1}{mn}\sum_{i=1,j=1}^{m,n} \norm{\nabla f_{i,j}(\xb_i^t) - \nabla f_{i,j}(x^*)}^2
	\\
	\le&
	\frac{2}{mn}\sum_{i=1,j=1}^{m,n} \left(\right.\norm{\nabla f_{i,j}(\xb_i^t) -\nabla f_{i,j}(\bbx^t)}^2\\
	&+\norm{\nabla f_{i,j}(\bbx^t)- \nabla f_{i,j}(x^*) }^2\left.\right)
	\\
	\le&
	\frac{2}{mn}\sum_{i=1,j=1}^{m,n} \norm{\nabla f_{i,j}(\bbx^t) - \nabla f_{i,j}(x^*)}^2
	\\
	&+
	\frac{2L^2}{m}\norm{\xb^t - \mathbf{1}\bbx^t}^2
    \\
	\overset{\eqref{eq:str_cvx}}{\le}&
	\frac{4L}{mn}\sum_{i=1,j=1}^{m,n} D_{f_{i,j}}(\bbx^t, x^*)
	+
	\frac{2L^2}{m}\norm{\xb^t - \mathbf{1}\bbx^t}^2
	\\
	=&
	4L \cdot D_f(\bbx^t, x^*)
	+
	\frac{2L^2}{m}\norm{\xb^t - \mathbf{1}\bbx^t}^2,
	\end{aligned}
    \end{equation*}
	where the second inequality is because of $L$-smoothness of $f_{i,j}$. If we use $L$-smoothness instead of Eqn. \eqref{eq:str_cvx} in the last inequality, we can obtain that
	\begin{equation*}
	\begin{aligned}
	&\frac{1}{mn}\sum_{i=1,j=1}^{m,n} \norm{\nabla f_{i,j}(\xb_i^t) - \nabla f_{i,j}(x^*)}^2\\
	\le&
	2L^2\norm{\bbx^t - x^*}^2
	+
		\frac{2L^2}{m}\norm{\xb^t - \mathbf{1}\bbx^t}^2.
	\end{aligned}
	\end{equation*}
\end{proof}

\begin{lemma} \label{lem:decomposition_v}
For \texttt{PMGT-SAGA}, it holds that
        \begin{equation}
        \begin{aligned}
		\label{eq:v_var_dsaga}
		&\frac{1}{m}\sum_{i=1}^m\EE\left[\norm{\vb_i^t - \nabla f_i(x^*)}^2\right]\\
		\le&
		\frac{2}{mn}\sum_{i=1,j=1}^{m,n} 
		\left(\right.\norm{\nabla f_{i,j}(\xb_i^t) - \nabla f_{i,j}(x^*)}^2\\
		&+
		\norm{\nabla f_{i,j}(x^*) - \nabla f_{i,j}(\phi_{i,j}^{t})}^2
		\left.\right),
	    \end{aligned}
	   \end{equation}

For \texttt{PMGT-LSVRG}, it holds that
        \begin{equation}
        \begin{aligned}
		\label{eq:v_var_dlsvrg}
		&\frac{1}{m}\sum_{i=1}^m\EE\left[\norm{\vb_i^t - \nabla f_i(x^*)}^2\right]\\
		\le&
		\frac{2}{mn}\sum_{i=1,j=1}^{m,n} 
		\left(\right.\norm{\nabla f_{i,j}(\xb_i^t) - \nabla f_{i,j}(x^*)}^2\\
		&+
		\norm{\nabla f_{i,j}(x^*) - \nabla f_{i,j}(\wb^{t}_i)}^2\left.
		\right).
	    \end{aligned}
	    \end{equation}

\end{lemma}
\begin{proof}
    For \texttt{PMGT-SAGA}, we have
    	\begin{align*}
	    &\frac{1}{m}\sum_{i=1}^m\EE\norm{\vb_i^t - \nabla f_i(x^*)}^2
		\\
		=&\frac{1}{m}\sum_{i=1}^m\EE_{j_i}\|\nabla f_{i,j_i}(\xb_i^t) - \nabla f_{i,j_i}(x^*) \\
		&+ \nabla f_{i,j_i}(x^*) - \nabla f_{i,j_i}(\phi_{i,j_i}^{t}) \\
		&+ \frac{1}{n}\sum_{j=1}^n \nabla f_{i,j}(\phi_{i,j}^{t}) - \nabla f_i(x^*)\|^2
		\\
		\le&
		\frac{2}{m}\sum_{i=1}^m\EE_{j_i}\left[\norm{\nabla f_{i,j_i}(\xb_i^t) - \nabla f_{i,j_i}(x^*)}^2\right]
		\\
		&+
		\frac{2}{m}\sum_{i=1}^m\EE_{j_i}\bigg[
		\Big\| \nabla f_{i,j_i}(x^*) - \nabla f_{i,j_i}(\phi_{i,j_i}^{t}) \\
		&~~- \EE_{j_i}\big[ \nabla f_{i,j_i}(x^*) - \nabla f_{i,j_i}(\phi_{i,j_i}^{t}) \big]\Big\|^2\bigg]
		\\
		\le&
		\frac{2}{m}\sum_{i=1}^m\EE_{j_i}\left[\norm{\nabla f_{i,j_i}(\xb_i^t) - \nabla f_{i,j_i}(x^*)}^2\right] \\
		&+
		\frac{2}{m}\sum_{i=1}^m\EE_{j_i} \left[\norm{ \nabla f_{i,j_i}(x^*) - \nabla f_{i,j_i}(\phi_{i,j_i}^{t})}^2 \right]
		\\
		\le&
		\frac{2}{m}\sum_{i=1}^m \EE_{j_i}\Big[\norm{\nabla f_{i,j_i}(\xb_i^t) - \nabla f_{i,j_i}(x^*)}^2\\
		&+
		\norm{\nabla f_{i,j_i}(x^*) - \nabla f_{i,j_i}(\phi_{i,j_i}^{t})}^2
		\Big]
		\\
		=&
		\frac{2}{mn}\sum_{i=1,j=1}^{m,n} 
		\norm{\nabla f_{i,j}(\xb_i^t) - \nabla f_{i,j}(x^*)}^2\\
		&+
		\norm{\nabla f_{i,j}(x^*) - \nabla f_{i,j}(\phi_{i,j}^{t})}^2, 
	\end{align*}
where the second inequality is because of $\EE[\norm{a - \EE[a]}^2] \le \EE[\norm{a}^2]$. For \texttt{PMGT-LSVRG}, we replace $\phi_{i,j}^{t}$ with $\wb_i^{t}$. This concludes the proof.
\end{proof}
We are ready to derive an upper bound for $\norm{\vb^{t+1}-\vb^t}^2$.
\begin{proof}[Proof of Lemma \ref{lem:local_gradient_estimator}]
    We first decompose $\norm{\vb^{t+1}-\vb^t}^2$:
	\begin{align*}
		&\norm{\vb^{t+1} - \vb^t}^2\\ 
		=&\sum_{i=1}^m \norm{\vb_i^{t+1} - \nabla f_i(x^*) - (\vb_i^t - \nabla f_i (x^*))  }^2
		\\
		\le&
		2\sum_{i=1}^m\left(\norm{\vb_i^{t+1} - \nabla f_i(x^*)}^2 + \norm{\vb_i^t - \nabla f_i (x^*)}^2\right).
	\end{align*}
	Then, for \texttt{PMGT-SAGA}, we have
	\begin{align*}
		&\EE\left[\frac{1}{m}\norm{\vb^{t+1} - \vb^t}^2\right]\\
		\le&
		\frac{2}{m}\sum_{i=1}^m\EE\left[\norm{\vb_i^{t+1} - \nabla f_i(x^*)}^2 + \norm{\vb_i^t - \nabla f_i (x^*)}^2\right]
		\\
		\overset{\eqref{eq:v_var_dsaga}}{\le}&
		\frac{4}{mn}\sum_{i=1,j=1}^{m,n} 
		\left(\norm{\nabla f_{i,j}(\xb_i^{t+1}) - \nabla f_{i,j}(x^*)}^2\right.\\
		&+\norm{\nabla f_{i,j}(x^*) - \nabla f_{i,j}(\phi_{i,j}^{t+1})}^2 \\
		&+\norm{\nabla f_{i,j}(\xb_i^t) - \nabla f_{i,j}(x^*)}^2\\
		&\left.+
		\norm{\nabla f_{i,j}(x^*) - \nabla f_{i,j}(\phi_{i,j}^{t})}^2
		\right)
		\\
		\overset{\eqref{eq:bnd_2}}{\le}&
		8L^2\norm{\bbx^{t+1} -x^*}^2 + \frac{8L^2}{m}\norm{\xb^{t+1} - \mathbf{1}\bbx^{t+1}}^2 + 4\Delta^{t+1}
		\\&
		+
		8L^2\norm{\bbx^t -x^*}^2 + \frac{8L^2}{m}\norm{\xb^t - \mathbf{1}\bbx^t}^2 + 4\Delta^t
		\\
		\overset{\eqref{eq:xxx}}{\le}&
		\left(8\rho^2 + 1\right)\frac{8L^2}{m}\norm{\xb^t - \mathbf{1}\bbx^t}^2 
		+
		\frac{64\rho^2\eta^2 L^2}{m}\norm{\bs^t - \mathbf{1}\bbs^t}^2
		\\&
		+
		8L^2\norm{\bbx^{t+1} -x^*}^2 +4\Delta^{t+1}
		+
		8L^2\norm{\bbx^t -x^*}^2 + 4\Delta^t,
	\end{align*}
	Similarly, for \texttt{PMGT-LSVRG}, we replace $\phi_{i,j}^{t}, \phi_{i,j}^{t+1}$ with $\wb_i^t$ and $\wb_i^{t+1}$, respectively in the above proof and note that the definitions of gradient learning quantity $\Delta^t$ are different for two algorithms. This concludes the proof.
\end{proof}

\begin{proof}[Proof of Lemma \ref{lem:consensus_error}]
    For both \texttt{PMGT-SAGA} and \texttt{PMGT-LSVRG}, using the upper bound of $\norm{\vb^{t+1}-\vb^t}$ from Lemma \ref{lem:local_gradient_estimator}, we have
\begin{align*}
	&\EE\left[ \frac{\eta^2}{m} \norm{\bs^{t+1} - \mathbf{1}\bbs^{t+1}}^2\right]\\
	\le&
	2\rho^2\cdot\left(\left(64\rho^2\eta^2 L^2 + 1\right)\frac{\eta^2}{m}  \norm{\bs^t - \mathbf{1}\bbs^t}^2\right.
	\\
    &+(8\rho^2 +1)\frac{8L^2\eta^2}{m}\norm{\xb^t - \mathbf{1}\bbx^t}^2+8L^2\eta^2\norm{\bbx^{t+1} -x^*}^2 \\
    &\left.+4\eta^2\Delta^{t+1} + 8L^2\eta^2\norm{\bbx^t - x^*}^2+4\eta^2\Delta^t
	\right).
	\end{align*}
    Therefore, we obtain the linear system inequality as desired.
\end{proof}

\section{Proofs of Lemmas~\ref{lem:gradient_convergence},~\ref{lem:gradient_learning_quantity}, and~\ref{lem:V_dec}}
\begin{proof}[Proof of Lemma \ref{lem:gradient_convergence}]
        For \texttt{PMGT-SAGA}, it holds that 
    	\begin{align*}
		&\EE\left[\norm{\bbs_t - \nabla f(x^*)}^2\right]\\
		=&
		\EE \left[\norm{\frac{1}{m}\sum_i^m \left(\vb_i^t - \nabla f_i(x^*)\right)}^2\right]
		\\
		\le&
		\frac{1}{m}\sum_{i=1}^m \EE\left[\norm{\vb_i^t - \nabla f_i(x^*)}^2\right]
		\\
		\overset{\eqref{eq:v_var_dsaga}}{\le}&
		\frac{2}{mn}\sum_{i=1,j=1}^{m,n} 
		\left(\norm{\nabla f_{i,j}(\xb_i^t) - \nabla f_{i,j}(x^*)}^2\right.
		\\
		&\left.~~+
		\norm{\nabla f_{i,j}(x^*) - \nabla f_{i,j}(\phi_{i,j}^{t})}^2
		\right)
		\\
		\overset{\eqref{eq:bnd_1}}{\le}&
		8L \cdot D_f(\bbx^t, x^*)
		+
		2\Delta^t
		+
		\frac{4L^2}{m}\norm{\xb^t - \mathbf{1}\bbx^t}^2.
	\end{align*}
	For \texttt{PMGT-LSVRG}, we replace $\phi_{i,j}^{t}$ with $\wb_i^t$ and note that the definitions of gradient learning quantity $\Delta^t$ for two algorithms are different. This concludes the proof.
\end{proof}

\begin{proof}[Proof of Lemma \ref{lem:gradient_learning_quantity}]
    For \texttt{PMGT-SAGA}, it holds that
    \begin{equation*}
	\begin{aligned}
	&\EE\left[\Delta^{t+1}\right]\\
	=&
	\frac{1}{mn}\sum_{i=1,j=1}^{m,n} \EE\left[\norm{\nabla f_{i,j}(\phi_{i,j}^{t+1}) - \nabla f_{i,j}(x^*)}^2\right] 
	\\
	=&\frac{1}{mn}\sum_{i=1,j=1}^{m,n}
	\left(
	\frac{1}{n}\norm{\nabla f_{i,j}(\xb_i^t) - \nabla f_{i,j}(x^*)}^2 
	\right.
	\\
	&\left.~~+
	\frac{n-1}{n}\norm{\nabla f_{i,j}(\phi_{i,j}^{t}) - \nabla f_{i,j}(x^*)}^2
	\right)
	\\
	=&
	\left(1 - \frac{1}{n}\right)\Delta^t 
	+
	\frac{1}{mn^2} \sum_{i=1,j=1}^{m,n} \norm{\nabla f_{i,j}(\xb_i^t) - \nabla f_{i,j}(x^*)}^2
	\\
	\overset{\eqref{eq:bnd_1}}{\le}&
	\left(1 - \frac{1}{n}\right)\Delta^t 
	+
	\frac{4L}{n} D_f(\bbx^t, x^*) 
	+
	\frac{2L^2}{mn}\norm{\xb^t - \mathbf{1}\bbx^t}^2.
	\end{aligned}
	\end{equation*}

	For \texttt{PMGT-LSVRG}, it holds that
	\begin{align*}
	&\EE\left[\Delta^{t+1}\right]\\
	=&
    \frac{1}{m n} \sum_{i=1, j=1}^{m, n} \EE\left[\norm{\nabla f_{i, j}\left(\wb_i^{t+1}\right)-\nabla f_{i, j}\left(x^{*}\right)}^{2}\right]	\\
	=&
	(1-p)\Delta^{t} + \frac{p}{mn} \sum_{i=1, j=1}^{m, n} \EE\left[\left\|\nabla f_{i, j}\left(\xb_i^{t}\right)-\nabla f_{i, j}\left(x^{*}\right)\right\|^{2}\right]
	\\
	\overset{\eqref{eq:bnd_1}}{\le}&
	\left(1-p\right) \Delta^{t}+4 Lp D_{f}\left(\bar{x}^{t}, x^{*}\right)+\frac{2 L^{2}p}{m}\left\|\mathbf{x}^{t}-\mathbf{1} \bar{x}^{t}\right\|^{2}.
	\end{align*}
	By replacing $p$ with $\frac{1}{n}$, we conclude the proof.
\end{proof}

\begin{proof}[Proof of Lemma \ref{lem:V_dec}]
	Let all expectations be conditioned on $\xb^t$ in this proof. We first note that the induction inequalities of two algorithms are of the same form. Therefore, the following proof works for both algorithms. By definitions of $\bs^t$ and $\bbs^t$, we have
	\begin{equation}
	\label{eq:xx}
	\begin{aligned}
	&\norm{\bbx^{t+1} - x^*}^2
	\\
	=&
	\left\|\frac{1}{m} \mathbf{1}^\top \proximal_{m\eta,R}(\xb^t - \eta \bs^t) -\proximal_{\eta  ,r}(\bbx^t - \eta\bbs^t)\right.
	\\
	&\left.+\proximal_{\eta  ,r}(\bbx^t - \eta\bbs^t)- \proximal_{\eta  ,r}(x^* - \eta \nabla f(x^*))\right\|^2
	\\
	=&
	\norm{\frac{1}{m} \mathbf{1}^\top \proximal_{m\eta,R}(\xb^t - \eta \bs^t) -\proximal_{\eta  ,r}(\bbx^t - \eta\bbs^t)}^2
	\\
	&+ 
	\norm{\proximal_{\eta  ,r}(\bbx^t - \eta\bbs^t)- \proximal_{\eta  ,r}(x^* - \eta \nabla f(x^*))}^2
	\\
	&
	+2 \left\langle\proximal_{\eta  ,r}(\bbx^t - \eta\bbs^t)- \proximal_{\eta  ,r}(x^* - \eta \nabla f(x^*)),\right.
	\\
	&\left.~~~~\frac{1}{m} \mathbf{1}^\top \proximal_{m\eta,R}(\xb^t - \eta \bs^t) -\proximal_{\eta  ,r}(\bbx^t - \eta\bbs^t)\right\rangle
	\\
    \overset{\eqref{eq:prox_diff}}{\le}&
    \underbrace{\norm{\bbx^t - x^* - \eta(\bbs^t - \nabla f(x^*))}^2}_{T^2}\\ 
	&+
	\frac{2}{m}\norm{\xb^t - \mathbf{1}\bbx^t}^2
	+\frac{2\eta^2}{m}\|\mathbf{s}^t-\mathbf{1}\bar{\mathbf{s}}^t\|^2
	\\
	&
	+\frac{2}{\sqrt{m}}\underbrace{\norm{\bbx^t - x^* - \eta(\bbs^t - \nabla f(x^*))}}_T
	\\
	&\cdot(\norm{\xb^t - \mathbf{1}\bbx^t}+\eta\|\mathbf{s}^t-\mathbf{1}\bar{\mathbf{s}}^t\|).
	\end{aligned}
	\end{equation}
	Next, we will bound the variance of $T$:
	\begin{align*}
	&\EE\left[T^2\right]\\
	=&
	\EE\left[\norm{\bbx^t - x^*}^2 - 2\eta\dotprod{\bbx^t - x^*, \bbs^t - \nabla f(x^*)} \right.
	\\
	&\left.+ \eta^2 \norm{\bbs^t - \nabla f(x^*)}^2\right]
	\\
	=&
	\norm{\bbx^t - x^*}^2
	-
	2\eta\dotprod{\EE[\bbs_t] - \nabla f(x^*), \bbx_t - x^*}\\
	&+
	\eta^2\EE\left[\norm{\bbs^t - \nabla f(x^*)}^2\right]
	\\
	\le&
	\norm{\bbx^t - x^*}^2
	-
	2\eta\dotprod{\nabla f(\bbx^t) - \nabla f(x^*), \bbx^t - x^*}\\
	&+
	2\eta\norm{\nabla f(\bbx^t) - \EE[\bbs^t]}\norm{\bbx^t - x^*}\\
	&+
	\eta^2\EE\left[\norm{\bbs^t - \nabla f(x^*)}^2\right]
	\\
	\overset{\eqref{eq:str_cvx}}{\le}&
	(1 - \eta \mu ) \norm{\bbx^t - x^*}^2 - 2\eta D_f(\bbx^t, x^*)\\
	&+
	2\eta\norm{\nabla f(\bbx^t) - \EE[\bbs^t]}\norm{\bbx^t - x^*}
	\\&
	+
	\eta^2\EE\left[\norm{\bbs^t - \nabla f(x^*)}^2\right]
	\\
	\overset{\eqref{eq:Var_}}{\le}&
	(1 - \eta \mu ) \norm{\bbx^t - x^*}^2
	+2\eta(4L\eta - 1) \cdot D_f(\bbx^t, x^*)\\
	&+2\eta^2\Delta^t+2\eta\norm{\nabla f(\bbx^t) - \EE[\bbs^t]}\norm{\bbx^t - x^*}\\
	&+ \frac{4L^2\eta^2}{m} \norm{\xb^t - \mathbf{1}\bbx^t}^2
	\\
	\overset{\eqref{eq:var_s}}{\le}&
	(1 - \eta \mu ) \norm{\bbx^t - x^*}^2
	+2\eta(4L\eta - 1) \cdot D_f(\bbx^t, x^*)\\
	&+2\eta^2\Delta^t 
	+\frac{2\eta L}{\sqrt{m}} \norm{\xb^t - \mathbf{1}\bbx^t}\cdot\norm{\bbx^t - x^*}\\
	&+ \frac{4L^2\eta^2}{m} \norm{\xb^t - \mathbf{1}\bbx^t}^2.
	\end{align*}
	Using the fact $\norm{\bbx^t - x^*} \le \sqrt{V^t}$, we can obtain that
	\begin{equation}
	\label{eq:ET}
	\begin{aligned}
	&\EE\left[T^2\right]\\ 
	\le&
	(1 - \eta \mu ) \norm{\bbx^t - x^*}^2
	+2\eta(4L\eta - 1) \cdot D_f(\bbx^t, x^*)\\
	&+2\eta^2\Delta^t 
	+\frac{2\eta L\sqrt{V^t}}{\sqrt{m}}\norm{\xb^t - \mathbf{1}\bbx^t}\\
	&+ \frac{4L^2\eta^2}{m} \norm{\xb^t - \mathbf{1}\bbx^t}^2,
	\end{aligned}
	\end{equation}
and	
    \begin{equation}
    \label{eq:T2Delta}
	\begin{aligned}
	&\EE\left[T^2 + 4n\eta^2\Delta^{t+1}\right]
	\\
	\overset{\eqref{eq:ET},\eqref{eq:Delta}}{\le}&
	(1 - \eta \mu ) \norm{\bbx^t - x^*}^2
	+
	2\eta(12L\eta - 1) D_f(\bbx^t, x^*)\\
	&+
    \frac{12L^2\eta^2}{m} \norm{\xb^t - \mathbf{1}\bbx^t}^2	+\frac{2\eta L\sqrt{V^t}}{\sqrt{m}}\norm{\xb^t - \mathbf{1}\bbx^t}\\
	&+ \left(1 - \frac{1}{n} + \frac{1}{2n}\right) (4n\eta^2\Delta^t) 
	\\
	\le&
	\max\left(1-\frac{1}{12\kappa}, 1 - \frac{1}{2n}\right) \cdot V^t
	+\frac{\sqrt{V^t}}{6\sqrt{m}}\norm{\xb^t - \mathbf{1}\bbx^t}
	\\
	&+ \frac{1}{12m} \norm{\xb^t - \mathbf{1}\bbx^t}^2,
	\end{aligned}
	\end{equation}
	where the last inequality is because we set $\eta = 1/(12L)$.
	
	Combining with Eqn.~\eqref{eq:xx} and the definition of Lyapunov function $V^t$, we have
	\begin{align*}
	&\EE[V^{t+1}]\\ 
	=& 
	\EE\left[\norm{\bbx^{t+1} - x^*}^2+ 4n\eta^2\Delta^{t+1}\right]
	\\
	\overset{\eqref{eq:xx}}{\le}&
	\mathbb{E}\left[T^{2}+4 n \eta^{2} \Delta^{t+1}\right]\\
	&+\frac{2}{\sqrt{m}} \sqrt{\mathbb{E}\left[T^{2}\right]} \cdot(\eta\|\mathbf{s}^t-\mathbf{1}\bar{\mathbf{s}}^t\|+\|\mathbf{x}^{t}-\mathbf{1} \bar{x}^{t}\|)\\
	&+\frac{2}{m}\left\|\mathbf{x}^{t}-\mathbf{1} \bar{x}^{t}\right\|^{2}+\frac{2\eta^2}{m}\|\mathbf{s}^t-\mathbf{1}\bar{\mathbf{s}}^t\|^2
	\\
    \overset{\eqref{eq:T2Delta}}{\le}&
     \max \left(1-\frac{1}{12 \kappa}, 1-\frac{1}{2 n}\right) \cdot V^{t}+\frac{\sqrt{V^{t}}}{6 \sqrt{m}}\left\|\mathbf{x}^{t}-\mathbf{1} \bar{x}^{t}\right\|\\
    &+\frac{2}{\sqrt{m}} \sqrt{\mathbb{E}\left[T^{2}+4 n \eta^{2} \Delta^{t+1}\right]} \cdot(\eta\|\mathbf{s}^t-\mathbf{1}\bar{\mathbf{s}}^t\|+\|\mathbf{x}^{t}-\mathbf{1} \bar{x}^{t}\|) \\
    &+
    \frac{25}{12 m}\left\|\mathbf{x}^{t}-\mathbf{1} \bar{x}^{t}\right\|^{2}+\frac{2\eta^2}{m}\|\mathbf{s}^t-\mathbf{1}\bar{\mathbf{s}}^t\|^2
    \\
    \leq & \max \left(1-\frac{1}{12 \kappa}, 1-\frac{1}{2 n}\right) \cdot V^{t}+\frac{\sqrt{V^{t}}}{6 \sqrt{m}}\left\|\mathbf{x}^{t}-\mathbf{1} \bar{x}^{t}\right\|\\
    &+\frac{2}{\sqrt{m}} \sqrt{\frac{13}{12} V^{t}+\frac{1}{6 m}\left\|\mathbf{x}^{t}-\mathbf{1} \bar{x}^{t}\right\|^{2}}\\
    &\cdot(\left\|\mathbf{x}^{t}-\mathbf{1} \bar{x}^{t}\right\|+\eta\|\mathbf{s}^t-\mathbf{1}\bar{\mathbf{s}}^t\|)
    \\
    &+
    \frac{25}{12 m}\left\|\mathbf{x}^{t}-\mathbf{1} \bar{x}^{t}\right\|^{2}+\frac{2\eta^2}{m}\|\mathbf{s}^t-\mathbf{1}\bar{\mathbf{s}}^t\|^2
    \\
    \leq &  \max \left(1-\frac{1}{12 \kappa}, 1-\frac{1}{2 n}\right) \cdot V^{t}\\
    &+\frac{3\sqrt{V^{t}}}{ \sqrt{m}}(\left\|\mathbf{x}^{t}-\mathbf{1} \bar{x}^{t}\right\|+\left\|\mathbf{s}^{t}-\mathbf{1} \bar{s}^{t}\right\|)\\
    &+(\frac{25}{12m}+\frac{2}{\sqrt 6m})\left\|\mathbf{x}^{t}-\mathbf{1} \bar{x}^{t}\right\|^{2}+\frac{2\eta^2}{m}\|\mathbf{s}^t-\mathbf{1}\bar{\mathbf{s}}^t\|^2\\
    &+\frac{2\eta}{\sqrt 6 m}\|\mathbf{x}^{t}-\mathbf{1} \bar{x}^{t}\|\|\mathbf{s}^t-\mathbf{1}\bar{\mathbf{s}}^t\|
    \\ 
    \leq &  \max \left(1-\frac{1}{12 \kappa}, 1-\frac{1}{2 n}\right) \cdot V^{t}\\
    &+\frac{3\sqrt{V^{t}}}{ \sqrt{m}}(\left\|\mathbf{x}^{t}-\mathbf{1} \bar{x}^{t}\right\|+\left\|\mathbf{s}^{t}-\mathbf{1} \bar{s}^{t}\right\|)\\
    &+(\frac{25}{12m}+\frac{2}{\sqrt 6m}+\frac{1}{3m})\left\|\mathbf{x}^{t}-\mathbf{1} \bar{x}^{t}\right\|^{2}\\
    &+(\frac{2\eta^2}{m}+\frac{\eta^2}{2m})\|\mathbf{s}^t-\mathbf{1}\bar{\mathbf{s}}^t\|^2\\
    \leq & \max \left(1-\frac{1}{12 \kappa}, 1-\frac{1}{2 n}\right) \cdot V^{t}\\
    &+\frac{3\sqrt{V^t}}{\sqrt{m}}(\norm{\xb^t - \mathbf{1}\bbx^t}+\eta \norm{\bs^t-\mathbf{1}\bar{\bs}^t})\\
    &+\frac{3.3}{m}\left\|\mathbf{x}^{t}-\mathbf{1} \bar{x}^{t}\right\|^{2}+\frac{2.5\eta^2}{m}\|\mathbf{s}^t-\mathbf{1}\bar{\mathbf{s}}^t\|^2.
	\end{align*}
	The third inequality is because Eqn.~\eqref{eq:T2Delta} and $\frac{\sqrt{V^t}}{6\sqrt{m}}\norm{\xb^t - \mathbf{1}\bbx^t} \le \frac{1}{12}V^t+\frac{1}{12m}\norm{\xb^t-\mathbf{1}\bbx^t}^2$. The forth inequality is because we apply $\sqrt{a^2+b^2} \leq |a|+|b|$ to $\sqrt{\frac{13}{12} V^{t}+\frac{1}{6 m}\left\|\mathbf{x}^{t}-\mathbf{1} \bar{x}^{t}\right\|^{2}}$. The fifth inequality is because $2ab\le a^2 + b^2$.
\end{proof}
\ifCLASSOPTIONcompsoc

\end{document}